\documentclass[11pt]{amsart}
\usepackage{amssymb}

\usepackage{amssymb,amsmath}

\def\nc{\newcommand}

 \def\Om{\Omega}
\def\ra{\rightarrow}

\nc\pa{\partial}

\nc\CC{\mathbb{C}}
\nc\RR{\mathbb{R}}
\nc\QQ{\mathbb{Q}}
\nc\ZZ{\mathbb{Z}}
\nc\NN{\mathbb{N}}
\newcommand{\Rn}{\mathbb R^n}
\newcommand{\gp}{\Gamma_p}
\newcommand{\Hn}{\mathbb H^n}

\nc\m[1]{\left| #1\right|}
\nc\norm[1]{\left\| #1\right\|}
\newcommand{\bG}{\bf{G}}

\newcommand{\p}{\partial}

\newtheorem{theorem}{Theorem}[section]
\newtheorem{lemma}[theorem]{Lemma}
\newtheorem{corollary}[theorem]{Corollary}

\newtheorem{definition}[theorem]{Definition}
\newtheorem{remark}[theorem]{Remark}        
\numberwithin{equation}{section}

\date{\today}

\begin{document}

\title[Inequalities of Hardy-Sobolev type in CC spaces]
{Inequalities of Hardy-Sobolev type in Carnot-Carath\'eodory spaces}

\author{Donatella Danielli}
\address{Department of Mathematics,
Purdue University, 150 N. University Street, West Lafayette, IN
47906, USA} \email{danielli@math.purdue.edu}
\thanks{First author supported in part by NSF CAREER Grant, DMS-0239771}

\author{Nicola Garofalo}
\address{Department of Mathematics,
Purdue University, 150 N. University Street,
West Lafayette, IN  47906, USA}
\email{garofalo@math.purdue.edu}
\thanks{Second author supported in part by NSF Grant DMS-07010001}

\author{Nguyen Cong Phuc}
\address{Department of Mathematics, Purdue University, 150 N. University Street,
West Lafayette, IN  47906, USA}
\email{pcnguyen@math.purdue.edu}

\begin{abstract}
We consider various types of Hardy-Sobolev
inequalities on a Carnot-Carath\'eodory space $(\Om, d)$ associated
to a system of smooth vector fields $X=\{X_1, X_2,\dots,X_m\}$ on
$\RR^n$ satisfying the H\"ormander's finite rank condition  $rank\
Lie[X_1,...,X_m] \equiv n$. One of our main concerns is the trace
inequality
\begin{equation*}
\int_{\Om}|\varphi(x)|^{p}V(x)dx\leq
C\int_{\Om}|X\varphi|^{p}dx,\qquad \varphi\in C^{\infty}_{0}(\Om),
\end{equation*}
where $V$ is a general weight, i.e., a nonnegative locally
integrable function on $\Om$, and $1<p<+\infty$. Under sharp
geometric assumptions on the domain $\Om\subset \Rn$ that can be measured equivalently
in terms of subelliptic capacities or Hausdorff contents,  we establish
various forms of Hardy-Sobolev  type inequalities.
\end{abstract}

\maketitle

\section{Introduction}
A celebrated inequality of S. L. Sobolev \cite{So} states that for
any $1<p<n$ there exists a constant $S(n,p)>0$ such that for every
function $\varphi\in C^\infty_0(\Rn)$ \begin{equation}\label{sob}
\left(\int_{\Rn} |\varphi|^{\frac{np}{n-p}}
dx\right)^{\frac{n-p}{np}}\ \leq\ S(n,p)\ \left(\int_{\Rn}
|D\varphi|^p dx\right)^{\frac{1}{p}}\ .
\end{equation}

Such an inequality admits the following extension, see \cite{CKN}. For
$0 \leq s \leq p$ define the critical exponent relative to $s$ as
follows \[ p^*(s)\ =\ p\ \frac{n-s}{n-p}\ .
\]
Then for every $\varphi \in C^\infty_0(\Rn)$ one has
\begin{equation}\label{hs}
\left(\int_{\Rn} \frac{|\varphi|^{p^*(s)}}{|x|^s}
dx\right)^{\frac{1}{p^*(s)}} \leq
\left(\frac{p}{n-p}\right)^{\frac{s}{p^*(s)}}
S(n,p)^{\frac{n(p-s)}{p(n-s)}} \left(\int_{\Rn} |D \varphi|^p
dx\right)^{\frac{1}{p}}\ .
\end{equation}

In particular, when $s=0$, then \eqref{hs} is just the Sobolev
embedding \eqref{sob}, whereas for $s=p$ we obtain the Hardy
inequality
\begin{equation}\label{h}
\int_{\Rn} \frac{|\varphi|^{p}}{|x|^p} dx\ \leq\
\left(\frac{p}{n-p}\right)^p\ \int_{\Rn} |D \varphi|^p dx\ .
\end{equation}

The constant $\left(\frac{p}{n-p}\right)^p$ in the right-hand side
of \eqref{h} is sharp. If one is not interested in the best
constant, then \eqref{hs}, and hence \eqref{h}, follows immediately
by combining the generalized H\"older's inequality for weak $L^p$
spaces in \cite{Hu} with the Sobolev embedding \eqref{sob}, after
having observed that $|\cdot|^{-s}\in L^{\frac{n}{s},\infty}(\Rn)$
(the weak $L^{\frac{n}{s}}$ space).

Inequalities of Hardy-Sobolev type play a fundamental role in
analysis, geometry and mathematical physics, and there exists a vast
literature concerning them. Recently, there has been a growing
interest in such inequalities in connection with the study of linear
and nonlinear partial differential equations of subelliptic type and
related problems in CR and sub-Riemannian geometry. In this context
it is also of interest to study the situation in which the whole
space is replaced by a bounded domain $\Om$ and instead of a one
point singularity such as in \eqref{hs}, \eqref{h}, one has the
distance from a lower dimensional set. We will be particularly
interested in the case in which such set is the boundary $\p \Om$ of
the ground domain.

In this paper we consider various types of Hardy-Sobolev
inequalities on a Carnot-Carath\'eodory space $(\Om, d)$ associated
to a system of smooth vector fields $X=\{X_1, X_2,\dots,X_m\}$ on
$\RR^n$ satisfying the H\"ormander's finite rank condition
\cite{Hor}
\begin{equation}\label{frc}
rank\ Lie[X_1,...,X_m]\ \equiv\ n .
\end{equation}

Here $\Om$ is a connected, (Euclidean) bounded  open set in $\RR^n$,
and $d$ is the  Carnot-Carath\'eodory (CC hereafter) metric
generated by $X$. For instance, a situation of special geometric
interest is that when the ambient manifold is a nilpotent Lie group
whose Lie algebra admits a stratification of finite step $r\geq 1$,
see \cite{FS}, \cite{F2} and \cite{St2}. These groups are called
Carnot groups of step $r$. When $r>1$ such groups are non-Abelian,
whereas when $r=1$ one essentially has Euclidean $\Rn$ with its
standard translations and dilations.

For a function $\varphi\in C^1(\Om)$ we indicate with $X\varphi =
(X_1\varphi,...,X_m\varphi)$ its ``gradient" with respect to the
system $X$. One of our main concerns is the trace inequality
\begin{equation}\label{wei}
\int_{\Om}|\varphi(x)|^{p}V(x)dx\leq C\int_{\Om}|X\varphi|^{p}dx,\qquad \varphi\in C^{\infty}_{0}(\Om),
\end{equation}
where $V$ is a general weight, i.e., a nonnegative locally
integrable function on $\Om$, and $1<p<+\infty$. This includes Hardy
inequalities  of the form
\begin{equation}\label{har}
\int_{\Om} \frac{|\varphi(x)|^{p}}{\delta(x)^{p}} dx\leq
C\int_{\Om}|X\varphi|^{p}dx\ ,
\end{equation}
and
\begin{equation}\label{har1}
\int_{\Om}\frac{|\varphi(x)|^{p}}{d(x, x_0)^{p}} dx\leq
C\int_{\Om}|X\varphi|^{p}dx\ ,
\end{equation}
as well as the mixed form
\begin{equation}\label{har2}
\int_{\Om}\frac{|\varphi(x)|^{p}}{\delta(x)^{p-\gamma}
d(x,x_0)^{\gamma}} dx\leq C\int_{\Om}|X\varphi|^{p}dx\ .
\end{equation}

In \eqref{har} we have denoted by $\delta(x)=\inf\{d(x,y):
y\in\partial\Om\}$ the CC distance of $x$ from the boundary of
$\Om$, in \eqref{har1} we have let $x_0$ denote a fixed point in
$\Om$, whereas in \eqref{har2} we have let $0\leq\gamma\leq p$.

Our approach to the inequalities \eqref{har}-\eqref{har2} is based
on results on subelliptic capacitary and Fefferman-Phong type
inequalities in \cite{D2}, Whitney decompositions, and the so-called
pointwise Hardy inequality
\begin{equation}\label{pwh2}
|\varphi(x)|\leq C
\delta(x)\Big(\sup_{0<r\leq4\delta(x)}\frac{1}{|B(x,r)|}\int_{B(x,r)}|
X\varphi|^q dy\Big)^{\frac{1}{q}}\ ,
\end{equation}
where $1<q<p$. In \eqref{pwh2}, $B(x,r)$ denotes the CC ball
centered at $x$ of radius $r$.

We use the ideas  in \cite{Ha} and \cite{Lehr} to show that
\eqref{pwh2} is essentially equivalent to several conditions on the
geometry of the boundary of $\Om$, one of which is the uniform
$(X,p)$-fatness of $\RR^n\setminus\Om$, a generalization of that of
uniform $p$-fatness introduced in \cite{Le} in the Euclidean setting
(see Definition \ref{fatset} below). Inequality \eqref{pwh2} is also
equivalent to other thickness conditions of $\RR^n\setminus\Om$
measured in terms of a certain Hausdorff content which is introduced
in Definition \ref{sec}. For the precise statement of these results
we refer to Theorem \ref{summa}.

We stress here that the class of uniformly $(X,p)$-fat domains is
quite rich. For instance, when $\mathbf G$ is a Carnot group of step
$r=2$, then every (Euclidean) $C^{1,1}$ domain is uniformly
$(X,p)$-fat for every $p>1$, see \cite{CG} and \cite{MM}. On the
other hand, one would think that the Carnot-Carath\'eodory balls
should share this property, but it was shown in \cite{CG} that this
is not the case, since even in the simplest setting of the
Heisenberg group these sets fail to be regular for the Dirichlet
problem for the relevant sub-Laplacian.

We now discuss our results concerning the trace inequality
\eqref{wei}. In the Euclidean setting, a necessary and sufficient
condition on $V$ was found by Maz'ya in 1962 \cite{Ma1}; see also
\cite{Ma2}, Theorem 2.5.2. That is, inequality \eqref{wei} with the
standard Euclidean metric induced by $X=\{\frac{\p}{\p x_1},\dots,
\frac{\p }{\p x_n}\}$ holds if and only if
\begin{equation}\label{MA}
\sup_{\substack{K\subset\Om \\K {\rm ~compact}}}\frac{\int_{K}V(x) dx}{{\rm cap}_{p}(K, \Om)}<+\infty,
\end{equation}
where ${\rm cap}_{p}(K, \Om)$ is the $(X,p)$-capacity $K$ defined by
$${\rm cap}_{p}(K,\Om)=\inf\left\{\int_{\Om}|Xu|^{p}dx: u\in C_{0}^{\infty}(\Om), u\geq 1 {\rm ~on~} K \right\}.$$

Maz'ya's result was  generalized to the subelliptic setting by the
first named author in \cite{D2}. However, although Corollary 5.9 in
\cite{D2} implies that $V\in L^{\frac{Q}{p}, \infty}(\Om)$ is
sufficient for \eqref{wei}, which is the case of an isolated singularity as in \eqref{har1},
 the Hardy inequality \eqref{har} could
not be deduced directly from it since $\delta(\cdot)^{-p}\not\in L^{\frac{Q}{p}, \infty}(\Om)$.
 Here $1<p<Q$, where $Q$ is the local homogeneous dimension of
$\Om$ (see section \ref{Pre}).  On the other hand, in the Euclidean
setting the Hardy inequality \eqref{har} was established in
\cite{An}, \cite{Le} and \cite{W} (see also \cite{Mik} and
\cite{BMS} for other settings)
 under the assumption that  $\RR^n\setminus\Om$ is uniformly $p$-fat.

In this paper we combine  a ``localized" version of \eqref{MA} and
the uniform $(X,p)$-fatness of $\RR^n\setminus\Om$ to allow the
treatment of weights $V$ with singularities which are distributed
both inside and on the boundary of $\Om$. More specifically, we show
that if $\RR^n\setminus\Om$ is uniformly $(X, p)$-fat then
inequality \eqref{wei} holds if and only if
\begin{equation*}
\sup_{B\in\mathcal{W}}\sup_{\substack{ K \subset 2B\\ K\, {\rm compact}}}
\frac{\int_{K}V(x)dx}{{\rm cap}_{p}(K,\Om)}<+\infty,
\end{equation*}
where $\mathcal{W}=\{B_j\}$ is a Whitney decomposition of $\Om$ as
in Lemma \ref{Whitney} below (see Theorem \ref{GP}). In the
Euclidean setting this idea was introduced in \cite{HMV}. Moreover,
a localized version of Fefferman-Phong condition
\begin{equation*}
\sup_{B\in\mathcal{W}}\sup_{\substack{x\in 2B\\0<r<{\rm diam}(B)}}
\int_{B(x,r)}V(y)^{s}dy\ \leq\ C\ \frac{|B(x,r)|}{r^{sp}}
\end{equation*}
for some $s>1$,  is also shown to be sufficient for \eqref{wei} (see Theorem \ref{FPT}).

With these general results in hands, in Corollaries \ref{LQ} and
\ref{LQ1} we deduce the Hardy type inequalities \eqref{har},
\eqref{har1}, and \eqref{har2} for domains $\Om$ whose complements
are uniformly $(X,p)$-fat. Note that in \eqref{har1} and
\eqref{har2} one has to restrict the range of $p$ to $1<p<Q(x_0)$,
where $Q(x_0)$ is the homogeneous dimension at $x_0$ with respect to
the system $X$ (see section \ref{Pre}). It is worth mentioning that
in the Euclidean setting inequalities of the form \eqref{har2} were
obtained in \cite{DPT} but   only for more regular domains, say,
$C^{1,\alpha}$ domains or domains  that satisfy a uniform exterior
sphere condition. In closing we mention that our results are of a
purely metrical character and that, similarly to \cite{D2}, they can
be easily generalized to the case in which the vector fields are
merely Lipschitz continuous and they satisfy the conditions in
\cite{GN1}.

\section{Preliminaries}\label{Pre}

Let $X=\{X_{1}, \dots,X_{m}\}$ be a system of $C^\infty$ vector
fields in $\RR^n$, $n\geq 3$, satisfying H\"ormander's finite rank
condition \eqref{frc}. For any two points $x,y\in\RR^n$, a piecewise
$C^1$ curve $\gamma(t):[0,T]\rightarrow\RR^n$ is said to be
sub-unitary, with respect to the system of vector fields $X$, if for
every $\xi\in\RR^n$ and $t\in (0,T)$ for which $\gamma'(t)$ exists
one has
$$(\gamma'(t)\cdot\xi)^{2}\leq\sum_{i=1}^{m}(X_{i}(\gamma(t))\cdot\xi)^{2}.$$

We note explicitly that the above inequality forces $\gamma '(t)$ to
belong to the span of $\{X_1(\gamma (t)),...,$ $ X_m(\gamma (t))\}$.
The sub-unit length of $\gamma$ is by definition $l_s(\gamma)=T$.
Given $x, y\in \Rn$, denote by $\mathcal S_\Om(x,y)$ the collection
of all sub-unitary $\gamma:[0,T]\to \Om$ which join $x$ to $y$. The
accessibility theorem of Chow and Rashevsky, see \cite{Ra} and
\cite{Ch}, states that, given a connected open set $\Om\subset \Rn$,
for every $x,y\in \Om$ there exists $\gamma \in \mathcal
S_\Om(x,y)$. As a consequence, if we pose
\[
d_{\Om}(x,y)\ =\ \text{inf}\ \{l_s(\gamma)\mid \gamma \in \mathcal
S_\Om(x,y)\} ,
\]
we obtain a distance on $\Om$, called the
\emph{Carnot-Carath\'eodory \emph{(CC)} distance on $\Om$},
associated with the system $X$. When $\Om = \Rn$, we write $d(x,y)$
instead of $d_{\Rn}(x,y)$. It is clear that $d(x,y) \leq
d_\Om(x,y)$, $x, y\in \Om$, for every connected open set $\Om
\subset \Rn$. In \cite{NSW} it was proved that for every connected
$\Om \subset \subset \Rn$ there exist $C, \epsilon >0$ such that
\begin{equation}\label{CCeucl}
C\ |x - y|\ \leq d_\Om(x,y)\ \leq C^{-1}\ |x - y|^\epsilon ,
\quad\quad\quad x, y \in \Om .
\end{equation}

This gives $d(x,y)\ \leq C^{-1} |x - y|^\epsilon$, $x, y\in \Om$,
and therefore
\[
i: (\Rn, |\cdot|)\to (\Rn, d) \quad\quad\quad is\,\ continuous .
\]

It is easy to see that also the continuity of the opposite inclusion
holds \cite{GN1}, hence the metric and the Euclidean topology are
compatible. In particular, the compact sets with respect to either
topology are the same.

For $x\in \Rn$ and $r>0$, we let $B(x,r) = \{y\in \Rn\mid d(x,y) < r
\}$. The basic properties of these balls were established by Nagel,
Stein and Wainger in their seminal paper \cite{NSW}. Denote by
$Y_1,...,Y_l$ the collection of the $X_j$'s and of those commutators
which are needed to generate $\Rn$. A formal ``degree" is assigned
to each $Y_i$, namely the corresponding order of the commutator. If
$I = (i_1,...,i_n), 1\leq i_j\leq l$ is a $n$-tuple of integers,
following \cite{NSW} we let $d(I) = \sum_{j=1}^n deg(Y_{i_j})$, and
$a_I(x) = \text{det}\ (Y_{i_1},...,Y_{i_n})$. The
\emph{Nagel-Stein-Wainger polynomial} is defined by
\begin{equation}\label{la}
\Lambda(x,r)\ =\ \sum_I\ |a_I(x)|\ r^{d(I)}, \quad\quad\quad\quad r
> 0.
\end{equation}

For a given compact set $K\subset \RR^n$, we denote by
\begin{equation}\label{Q}
Q = \text{sup}\{d(I): |a_I(x)| \ne 0,\, x\in K\}
\end{equation}
the {\it local homogeneous dimension} of $K$ with respect to the system $X$, and by
\begin{equation}\label{Qx}
Q(x) = \text{inf} \{d(I): |a_I(x)| \ne 0\}
\end{equation}
the {\it homogeneous dimension} at $x$ with respect to $X$.
Obviously, $3 \leq n \leq Q(x) \leq Q$. It is immediate that for
every $x\in K$, and every $r>0$, one has
\begin{equation}\label{lambdarescale}
t^Q \Lambda(x,r) \leq \Lambda(x,tr) \leq  t^{Q(x)}
\Lambda(x,r)
\end{equation}
for any $0\leq t\leq 1$, and thus
\begin{equation}\label{nswhom}
Q(x) \leq \frac{r \Lambda'(x,r)}{\Lambda(x,r)} \leq Q\ .
\end{equation}

For a simple example consider in $\RR^3$ the system \[ X\ =\
\{X_1,X_2,X_3\}\ =\ \left\{\frac{\partial}{\partial
x_1},\frac{\partial}{\partial x_2},x_1\frac{\partial}{\partial
x_3}\right\}\ . \] It is easy to see that $l=4$ and \[
\{Y_1,Y_2,Y_3,Y_4\}\ =\ \{X_1,X_2,X_3,[X_1,X_3]\}\ . \] Moreover,
$Q(x) = 3$ for all $x\not= 0$, whereas for any compact set $K$
containing the origin $Q(0)=Q = 4$.

The following fundamental result is due to Nagel, Stein and Wainger
\cite{NSW}: \emph{For every compact set $K\subset\RR^n$ there exist
constants $C, R_{0}>0$ such that, for any $x\in K$, and $0 < r \leq
R_{0}$, one has
\begin{equation}\label{nsw2}
C \Lambda(x,r) \leq |B(x,r)| \leq C^{-1} \Lambda(x,r)\ .
\end{equation}}

As a consequence, there exists $C_0$ such that  for any $x\in K$,
$0< r <s\leq R_{0}$,  we have
\begin{equation}\label{doubling}
  C_{0}\left(\frac{r}{s}\right)^{Q} \leq \frac{|B(x,r)|}{ |B(x,s)|}\ .
\end{equation}

 Henceforth, the numbers $C_0$ and $R_0$ above will be referred to as the {\it local parameters }
of $K$ with respect to the system $X$. If $E$ is any (Euclidean)
bounded set in $\RR^n$ then the local parameters of $E$ are defined
as those of $\overline{E}$. We mention explicitly that the number
$R_0$ is always chosen in such a way that the closed metric balls
$\overline B(x,R)$, with $x\in K$ and $0<R\leq R_0$, are compact,
see \cite{GN1} and \cite{GN2}. This choice is motivated by the fact
that in a CC space the closed metric balls of large radii are not
necessarily compact. For instance, if one considers the H\"ormander
vector field on $\mathbb R$ given by $X_1 = (1+x^2)\frac{d}{dx}$,
then for any $R\geq\pi/2$ one has $B(0,R) = \mathbb R$, see
\cite{GN1}.

Given an open set $\Om\subset \Rn$, and $1\leq p\leq \infty$, we
denote by $S^{1,p}(\Om)$, the subelliptic Sobolev space associated
with the system $X$ is defined by
$$S^{1,p}(\Om)=\{u\in L^{p}(\Om): X_{i}u\in L^{p}(\Om), i=1,\dots,m\},$$
where $X_{i}u$ is understood in the distributional sense, i.e.,
$$ <X_{i}u,\varphi>=\int_{\Om}u X^*_{i}\varphi dx$$
for every $\varphi\in C^{\infty}_{0}(\Om)$. Here $X_i^{*}$ denotes the formal adjoint of $X_i$. Endowed with the norm
\begin{equation}\label{sobnorm}
\norm{u}_{S^{1,p}(\Om)}=\left(\int_{\Om}(|u|^{p}+|Xu|^{p})dx\right)^{\frac{1}{p}},
\end{equation}
$S^{1,p}(\Om)$ is a Banach space which admits $C^{\infty}(\Om)\cap
S^{1,p}(\Om)$ as a dense subset, see \cite{GN1} and \cite{FSS}. The
local version of $S^{1,p}(\Om)$ will be denoted by $S^{1,p}_{\rm
loc}(\Om)$, whereas the completion of $C^{\infty}_{0}(\Om)$ under
the norm in \eqref{sobnorm} is denoted by $S^{1,p}_{0}(\Om)$.

A fundamental result in \cite{RS} shows that, for any bounded open
set $\Om \subset \Rn$ the space $S^{1,p}_0(\Om)$ embeds into a
standard fractional Sobolev space $W^{s,p}_0(\Om)$, where $s = 1/r$
and $r$ is the largest number of commutators which are needed to
generate the Lie algebra over $\overline \Om$. Since on the other
hand we have classically $W^{s,p}_0(\Om)\subset L^p(\Om)$, we obtain
the following Poincar\'e inequality
\begin{equation}\label{poi}
\int_\Om |\varphi|^p\ dx\ \leq\ C(\Om)\ \int_\Om |X\varphi|^p\ dx\
,\ \ \ \varphi \in S^{1,p}_0(\Om)\ .
\end{equation}

Another fundamental result which plays a pervasive role in this
paper is the following global Poincar\'e inequality on metric balls
due to D. Jerison \cite{J}. Henceforth, given a measurable set
$E\subset \Rn$, the notation $\varphi_E$ indicates the average of
$\varphi$ over $E$ with respect to Lebesgue measure.

\begin{theorem}\label{T:jer}
Let $K\subset \Rn$ be a compact set with local parameters $C_0$ and
$R_0$. For any $1\leq p <\infty$ there exists $C =C(C_0,p)>0$ such
that for any $x\in K$ and every $0<r\leq R_0$, one has for all
$\varphi\in S^{1,p}(B(x,r))$
\begin{equation}\label{poincare}
\int_{B(x,r)}|\varphi-\varphi_{B(x,r)}|^{p}dy\ \leq\ C\ r^p\
\int_{B(x,r)}|X\varphi|^{p}dy\ .
\end{equation}
\end{theorem}

We will also need the following basic result on the existence of
cut-off functions in metric balls, see \cite{GN2} and also
\cite{FSS}. Given a set $\Om \subset \Rn$ we will indicate with
$C^{0,1}_d(\Om)$ the collection of functions $\varphi\in C(\Om)$ for
which there exists $L\geq 0$ such that
\[
|\varphi(x) - \varphi(y)| \ \leq\ L\ d(x,y)\ ,\ \ \ x, y\in \Om\ .
\]

We recall that, thanks to the Rademacher-Stepanov type theorem
proved in \cite{GN2}, \cite{FSS}, if $\Om$ is metrically bounded
then any function in $C^{0,1}_d(\Om)$ belongs to the space
$S^{1,\infty}(\Om)$. This is true, in particular, when $\Om$ is a
metric ball.

\begin{theorem}\label{T:cutoff}
Let $K\subset \Rn$ be a compact set with local parameters $C_0$ and
$R_0$. For every $0<s<t<R_0$ there exists $\varphi\in
C^{0,1}_d(\Rn)$, $0\leq \varphi\leq 1$, such that
\begin{itemize}
\item[(i)] $\varphi\equiv 1$ on $B(x,s)$ and $\varphi\equiv 0$ outside $B(x,t)$,
\item[(ii)] $|X\varphi|\leq \frac{C}{t-s}$ for a.e. $x\in\Rn$\ ,
\end{itemize}
for some $C>0$ depending on $C_0$. Furthermore, we have $\varphi\in
S^{1,p}(\Rn)$ for every $1\leq p<\infty$.
\end{theorem}

A condenser is a couple $(K,\Om)$, where $\Om$ is open and $K\subset
\Om$ is compact. The subelliptic $p$-capacity of $(K,\Om)$ is
defined by
$${\rm cap}_{p}(K,\Om)=\inf\left\{\int_{\Om}|X\varphi|^{p}dx: \varphi\in
C_{d}^{0,1}(\Rn), supp\ \varphi \subset \Om, \varphi\geq 1 {\rm ~on~} K \right\}\ .$$

As usual, it can be extended to arbitrary sets $E\subset\Om$ by
letting
$${\rm cap}_{p}(E,\Om)=\inf_{\substack{G\subset\Om {\rm ~open}\\E\subset G}}
\sup_{\substack{K\subset G\\K {\rm ~compact}}} {\rm cap}_{p}(K,
\Om)\ .$$

It was proven in \cite{D1} that  the subelliptic $p$-capacity of a
metric ``annular" condenser has the following two-sided estimate
which will be used extensively in the paper. Given a compact set
$K\subset \Rn$ with local parameters $C_0$ and $R_0$, and
homogeneous dimension $Q$, for any $1<p<\infty$ there exist $C_1,
C_2>0$, depending only on $C_0$ and $p$, such that
\begin{equation}\label{anular}
C_{1} \frac{|B(x,r)|}{r^{p}}\ \leq\ {\rm cap}_{p}(B(x,r), B(x,2r))\
\leq\ C_{2}\ \frac{|B(x,r)|}{r^{p}}\ ,
\end{equation}
for all $x\in K$, and $0<r\leq R_{0}/2$.

The subelliptic $p$-Laplacian associated to the system $X$ is the
quasilinear operator defined by
$$\mathcal{L}_{p}[u]=-\sum_{i=1}^{m}X_{i}^{*}(|Xu|^{p-2}X_{i}u)\ .$$
A weak solution
$u\in S^{1,p}_{{\rm loc}}(\Om)$  to the equation
$\mathcal{L}_p[u]=0$ is said to be $\mathcal{L}_{p}$-$harmonic$ in
$\Om$. It is well-known that every $\mathcal{L}_{p}$-harmonic
function in $\Om$ has a H\"older continuous representative, see
\cite{CDG1}. This means that, if $C_0$ and $R_0$ are the local
parameters of $\Om$, then there exist $0<\alpha<1$, and $C>0$,
depending on $C_0$ and $p$, such that for every $0<R\leq R_0$ for
which $B_{4R}(x_0)\subset \Om$ one has
\begin{equation}\label{HC}
|u(x) - u(y)|\ \leq\ C \left(\frac{d(x,y)}{R}\right)^{\alpha}
\left(\frac{1}{|B_{2R}(x_0)|} \int_{B_{2R}(x_0)} |u|^p
dx\right)^{1/p}\ .
\end{equation}

Given a bounded open set $\Om \subset \Rn$ and $1<p<\infty$, the
Dirichlet problem for $\Om$ and $\mathcal L_p$ consists in finding,
for every given $\varphi\in S^{1,p}(\Om)\cap C(\overline \Om)$, a
function $u\in S^{1,p}(\Om)$ such that
\begin{equation}\label{DP}
\mathcal L_p[u]\ =\ 0\ \ \  \text{in}\ \Om\ ,\ \ \ \ \  u - \varphi
\in S^{1,p}_0(\Om)\ .
\end{equation}

Such problem admits a unique solution, see \cite{D1}. A point
$x_0\in \p \Om$ is called regular if for every $\varphi\in
S^{1,p}(\Om)\cap C(\overline \Om)$, one has $\underset{x\to
x_0}{\lim} u(x) = \varphi(x_0)$. If every $x_0\in \p \Om$ is
regular, then we say that $\Om$ is regular. We will need the
following basic Wiener type estimate proved in \cite{D1}.

\begin{theorem}\label{T:Wiener}
Given a bounded open set $\Om\subset \Rn$ with local parameters
$C_0$ and $R_0$, let $\varphi\in S^{1,p}(\Om)\cap C(\overline \Om)$.
Consider the (unique) solution $u$ to the Dirichlet problem
\eqref{DP}. There exists $C = C(p,C_0)>0$ such that given $x_0\in
\partial \Om$, and $0<r<R\leq R_0/3$, one has with $\Om^c = \Rn
\setminus \Om$
\begin{align*}
& osc \{u, \Om\cap B(x_0,r)\}\ \leq\ osc \{\varphi, \partial \Om
\cap \overline B(x_0,2R)\}
\\
& +\ osc (\varphi, \partial \Om)\ \exp\ \left\{-\ C\
\int_r^R\left[\frac{{\rm cap}_p\ (\Om^c\cap \overline B(x_0,t),
B(x_0,2t))}{{\rm cap}_p\ (\overline B(x_0,t), B(x_0,2t))}\right]\
\frac{dt}{t}\right\} .
\end{align*}
\end{theorem}

\begin{remark}\label{R:wiener}
It is clear from Theorem \ref{T:Wiener} that if $\Om$ is
\emph{thin at} $x_0\in \p \Om$, i.e., if one has
\[
\underset{t\to 0^+}{\liminf}\ \frac{{\rm cap}_p\ (\Om^c\cap \overline
B(x_0,t), B(x_0,2t))}{{\rm cap}_p\ (\overline B(x_0,t), B(x_0,2t))}\ >\ 0\
,
\]
then $x_0$ is regular for the Dirichlet problem \eqref{DP}.
\end{remark}

A lower semicontinuous function $u:\Om\ra (-\infty, \infty]$, such
that $u\not\equiv +\infty$, is called
$\mathcal{L}_{p}$-$superharmonic$ in $\Om$ if for all open sets  $D$
such that ${\overline D}\subset\Om$, and all
$\mathcal{L}_{p}$-harmonic functions $h\in C(\overline{D})$, the
inequality $h\leq u$ on $\partial D$ implies $h\leq u$ in $D$.
Similarly to what is done in the classical case in \cite{HKM}, one
can associate with each $\mathcal{L}_{p}$-superharmonic function $u$
in $\Om$ a nonnegative (not necessarily finite) Radon measure
$\mu[u]$, such that $- \mathcal{L}_{p}[u]=\mu[u]$. This means that
$$\int_{\Om}|Xu|^{p-2}Xu \cdot X\varphi\ dx\ =\ \int_{\Om}\varphi\ d\mu[u]$$
for all $\varphi\in C^{\infty}_{0}(\Om)$. Here $Xu$ is defined a.e. by
$$Xu=\lim_{k\rightarrow\infty}X(\min\{u, k\}).$$

It is known that, if either $u\in L^{\infty}(\Om)$ or $u\in
S^{1,r}_{\rm loc}(\Om)$ for some $r\geq 1$, then $Xu$ coincides with
the regular distributional derivatives. In general we have $Xu \in
L^{s}_{\rm loc}(\Om)$ for $0<s<\frac{Q(p-1)}{Q-1}$; see e.g.,
\cite{TW} and \cite{HKM}.

We will need  the following basic pointwise estimates for
$\mathcal{L}_p$-superharmonic functions. This result was first
established by Kilpel\"ainen and Mal\'y \cite{KM} in the elliptic
case, and extended to the setting of CC metrics by Trudinger and
Wang \cite{TW}. For a generalization to more general metric spaces
we refer the reader to \cite{BMS}. We recall that given $1<p<\infty$
the $p$-\emph{Wolff's potential} of a Radon measure $\mu$ on a
metric ball $B(x,R)$ is defined by
\begin{equation}\label{Wpot} {\rm\bf W}_{p}^{R}\mu(x)\ =\
\int_{0}^{R}\left[\frac{\mu(B(x,t))}{t^{-p}|B(x,t)|}\right]^{\frac{1}{p-1}}\frac{dt}{t}
\end{equation}

\begin{theorem}\label{wes}
Let $K\subset\RR^n$ be a compact set with relative local parameters
$C_0$ and $R_0$. If $x\in K$ and $R\leq R_0/2$, let $u \ge 0$ be
$\mathcal{L}_p$-superharmonic in $B(x,2R)$ with associated measure
$\mu= - \mathcal{L}_p[u]$. There exist positive constants $C_{1}$ and
$C_{2}$, depending only on $p$ and $C_0$, such that
$$C_1{\rm\bf W}_{p}^{R}\mu(x)\leq u(x)\leq C_2\left\{{\rm\bf W}_{p}^{2R}\mu(x) +\inf_{B(x,R)}u\right\}\ .$$
\end{theorem}

\section{Pointwise Hardy Inequalities }

We begin this section by generalizing a Sobolev type inequality
that, in the Euclidean setting, was found by Maz'ya in \cite{Ma2},
Chapter 10.

\begin{lemma}\label{SobCap} Let $K\subset\RR^n$ be a compact set with local parameters $C_0$ and
$R_0$, and for $x\in K$ and $r\leq R_{0}/2$, set $B=B(x,r)$. Given
$1\leq q<\infty$ there exists a constant $C>0$ depending only on
$C_0$ and $q$, such that for all $\varphi\in C^{\infty}(2B)$
\begin{equation}\label{capinequ}
|\varphi_{B}|\leq C \left(\frac{1}{{\rm cap}_{q}(\{\varphi=0\}\cap
\overline B, 2B)}\int_{2B}|X\varphi|^q dx \right)^{\frac{1}{q}}\ .
\end{equation}

\end{lemma}
\begin{proof}
We may assume that $\varphi_{B}\not=0$ for otherwise there is
nothing to prove. Let $\eta\in C^{0,1}_{d}(\Rn)$, $0\leq\eta\leq 1$,
$supp\ \eta \subset 2B$, $\eta=1$ on $\overline{B}$ and $|X\eta|\leq
\frac{C}{r}$, be a cut-off function as in Theorem \ref{T:cutoff}. Define
$\phi=\eta(\varphi_{B}-\varphi)/\varphi_{B}$, then $\phi\in
C^{0,1}_{d}(\Rn)$, $supp\ \phi \subset 2B$, and $\phi=1$ on
$\{\varphi=0\}\cap \overline{B}$. It thus follows that
\begin{eqnarray}\label{capleq}
\lefteqn{{\rm cap}_{q}(\{\varphi=0\}\cap \overline B, 2B)\leq \int_{2B}|X\phi|^{q}dx}\\
&\leq& |\varphi_{B}|^{-q} \int_{2B}|X\eta|^{q}|\varphi-\varphi_{B}|^{q}dx + |\varphi_{B}|^{-q}\int_{2B}|X\varphi|^{q}dx\nonumber\\
&\leq& C
|\varphi_{B}|^{-q}r^{-q}\int_{2B}|\varphi-\varphi_{B}|^{q}dx
+|\varphi_{B}|^{-q} \int_{2B}|X\varphi|^{q}dx\ .\nonumber
\end{eqnarray}
On the other hand, by Theorem \ref{T:jer} and by \eqref{doubling} we
infer
\begin{eqnarray*}
\int_{2B}|\varphi-\varphi_{B}|^{q}dx &\leq& C \int_{2B}|\varphi-\varphi_{2B}|^{q}dx +C \int_{2B}|\varphi_{B}-\varphi_{2B}|^{q}dx\\
&\leq& C r^{q}\int_{2B}|X\varphi|^{q}dx+ C\int_{2B}|\varphi-\varphi_{2B}|^{q}dx\\
&\leq& C r^{q}\int_{2B}|X\varphi|^{q}dx\ .
\end{eqnarray*}

Inserting the latter inequality in \eqref{capleq} we find
$${\rm cap}_{q}(\{\varphi=0\}\cap \overline B, 2B)\leq C |\varphi_{B}|^{-q}\int_{2B}|X\varphi|^{q}dx\ ,$$
which gives the desired inequality \eqref{capinequ}.

\end{proof}

We now introduce the notion of uniform $(X, p)$-fatness. As Theorem
\ref{summa} below proves, such notion turns out to be equivalent to
a pointwise Hardy inequality and to a uniform thickness property
expressed in terms of the Hausdorff content.

\begin{definition}\label{fatset}
We say that a set  $E\subset\RR^n$ is  uniformly $(X, p)$-fat with constants $c_0, r_0>0$ if
$${\rm cap}_{p}(E \cap \overline{B}(x,r), B(x,2r))\ \geq\ c_{0}\ {\rm cap}_{p}(\overline{B}(x,r), B(x,2r))$$
for all $x\in \partial E$ and for all $0<r\leq r_{0}$.
\end{definition}

The potential theoretic relevance of Definition \ref{fatset} is
underscored in Remark \ref{R:wiener}. From the latter it follows
that, if $\Rn\setminus \Om$ is uniformly $(X,p)$-fat, then for every
$x_0\in \p \Om$ one has for every $\varphi\in S^{1,p}(\Om)\cap
C(\overline \Om)$
\[
osc \{u, \Om\cap B(x_0,r)\}\ \leq\ osc \{\varphi, \partial \Om \cap
\overline B(x_0,2R)\}\ ,
 \]
and therefore $\Om$ is regular for the Dirichlet problem for the
subelliptic $p$-Laplacian $\mathcal L_p$.

Uniformly $(X, p)$-fat sets enjoy the following self-improvement
property which was discovered in \cite{Le} in the Euclidean setting.
Such property holds also in the setting of weighted Sobolev spaces
and degenerate elliptic equations \cite{Mik}. The proof in
\cite{Mik} uses the Wolff's potential and works also in the general
setting of metric spaces \cite{BMS}. For the sake of completeness,
we will include its details here.

\begin{theorem}\label{qthick} Let $\Om\subset\RR^n$ be a bounded domain with local
parameters $C_0$ and $R_0$. There exists a constant $0<r_0\leq
R_0/100$ such that  whenever $\RR^n\setminus\Om$ is uniformly $(X,
p)$-fat with constants $c_0$ and $r_0$, then it is also uniformly
$(X, q)$-fat for some $q<p$ with constants $c_1$ and $r_0$.
\end{theorem}

\begin{proof} Let ${\rm dist}(x, \Om)=\inf\{d(x,y): y\in\Om\}$ and denote by $U\subset \Rn$ the compact set
$$U=\{x\in\RR^n: {\rm dist}(x, \Om)\leq R_0\}\ ,$$
with local parameters $C_1, R_1$. We will show  that if
$\RR^n\setminus\Om$ is uniformly $(X, p)$-fat with constants $c_0$
and $r_0 =\min\{R_0, R_1\}/100$, then it is also uniformly $(X,
q)$-fat for some $q<p$ with constants $c_1$ and $r_0$. To this end,
we fix $x_0\in
\partial\Om$ and $0<R\leq r_0$.
 Following \cite{Le}, we first claim that there exists a compact set $K\subset
(\RR^n\setminus\Om)\cap \overline{B}(x_0,R)$ containing $x_0$ such
that $K$ is uniformly $(X, p)$-fat with constants $c_1>0$ and $R$.
Indeed, let
 $E_1=(\RR^n\setminus\Om)\cap B(x_0,\tfrac{R}{2})$ and inductively let
$$E_{k}=(\RR^n\setminus\Om)\cap \left(\bigcup_{x\in E_{k-1}}B(x,\tfrac{R}{2^k})\right),\quad k\in \mathbb N\ .$$
Then it is easy to see that $K$ can be taken as the closure of $\cup_{k}E_k$.

Let now $B=B(x_0,R)$ and denote by $\hat{P}_{K}$ the potential of
$K$ in $2B$, i.e., $\hat{P}_{K}$ is the lower semicontinuous
regularization
$$\hat{P}_{K}(x)=\lim_{r\rightarrow 0}\inf_{B_{r}(x)}P_{K}\ ,$$
where $P_{K}$ is defined by
$$P_{K}=\inf\{u: u {\rm ~is~} \mathcal{L}_p{\text -}{\rm superharmonic~in~} 2B, {\rm ~and~} u\geq \chi_{K} \}.$$

Let $\mu=-\mathcal{L}_{p}[\hat{P}_{K}]$ then  $ supp\ \mu\subset
\partial K$ and
\begin{equation}\label{equm}
\mu(K)= {\rm cap}_{p}(K, 2B)\ .
\end{equation}

Moreover, $\hat{P}_{K}=P_K$ except for a set of zero capacity ${\rm
cap}_p(\cdot, 2B)$ (see \cite{TW}). Hence $\hat{P}_{K}$ is the
unique solution in $S^{1,p}_{0}(2B)$ to the Dirichlet problem \[
\mathcal{L}_{p}[u]=0 \quad {\rm in}\quad 2B\setminus K\ ,\ \ \ \
u-f\in S^{1,p}_{0}(2B\setminus K)\ , \] for any $f\in
C^{\infty}_{0}(2B)$ such that  $f\equiv 1$ on $K$. Thus by Theorem
\ref{T:Wiener} and the $(X,p)$-fatness of $K$ there are constants
$C>0$, $\alpha >0$ independent of  $R$ such that
\begin{equation}\label{osces}
 osc\ (\hat{P}_{K}, B(x,r))\leq C R^{-\alpha} r^{\alpha}
\end{equation}
for all $x\in\partial K$ and $0<r\leq R/2$. From the lower Wolff's potential estimate in Theorem \ref{wes} we have
\begin{eqnarray*}
\left[\frac{\mu(B(x,r))}{r^{-p}|B(x,r)|}\right]^{\frac{1}{p-1}}&\leq&
C\,
{\rm\bf W}_p^{2r}\mu(x) \leq C\left(\hat{P}_{K}(x)-\inf_{B(x,4r)}\hat{P}_{K}\right)\\
&\leq& C\,  osc\ (\hat{P}_{K}, B(x,4r))\ .
\end{eqnarray*}

Thus it follows from \eqref{osces} that
\begin{equation}\label{balles}
\mu(B(x,r))\leq C R^{-\alpha(p-1)}r^{\alpha(p-1)-p}|B(x,r)|
\end{equation}
for all $x\in\partial K$ and $0<r\leq R/8$. Moreover, since
$supp\ \mu\subset\partial K$ we see from the doubling property
\eqref{doubling} that \eqref{balles} holds also for all $x\in
B(x_0,2R)$ and $0<r\leq R/16$. In fact, it then holds for all
$R/16<r\leq 3R$ as well since, again by \eqref{doubling}, the ball
$B(x,r)$ can be covered by a fixed finite number of balls of
radius $R/16$.

We next pick $q\in\RR$ such that $p-\alpha(p-1)<q<p$ and define a
measure $\nu=R^{p-q}\mu$. From \eqref{balles} it follows that for
all $x\in B(x_0,2R)$,
\begin{equation}\label{uwes}
{\rm\bf W}_{q}^{3R}\nu(x)\leq C
R^{\frac{p-q-\alpha(p-1)}{q-1}}\int_{0}^{3R}
r^{\frac{q-p+\alpha(p-1)}{q-1}} \frac{dr}{r}\leq M\ ,
\end{equation}
where $M$ is independent of $R$. Thus by Lemma 3.3 in \cite{B},
$\nu$ belongs to the dual space of $S^{1,q}_{0}(2B)$  and there is a unique
solution $v\in S^{1,q}_{0}(2B)$  to the problem
\begin{eqnarray}\label{equnu}
\left\{\begin{array}{c}
-\mathcal{L}_{q}[v]=\nu \quad {\rm in}\quad 2B\\
v=0 \quad {\rm on} \quad \partial(2B)\ .
\end{array}
\right.
\end{eqnarray}

We now claim that
\begin{equation}\label{claim}
v(x)\leq c
\end{equation}
for all $x\in 2B$ and for a constant $c$ independent of $R$. To this
end, it is enough to show \eqref{claim} only for  $x\in
\overline{B}$ since $v$ is $\mathcal{L}_{q}$-harmonic in
$2B\setminus\overline{B}$ and $v=0$ on $\partial(2B)$. Fix now $x\in
\overline{B}$. By Theorem \ref{wes} we have
\begin{equation}\label{wtheterm}
v(x)\leq C\left\{ {\rm\bf W}_{q}^{3R}\nu(x) +
\inf_{B(x,\tfrac{R}{4})}v\right\}\ .
\end{equation}

To bound the term $\inf_{B(x,\tfrac{R}{4})}v$ in \eqref{wtheterm}, we first
use $\min\{v, k\}$, $k>0$, as a test function in \eqref{equnu} to
obtain
\begin{eqnarray}\label{deres}
\int_{2B}|X(\min\{v,k\})|^qdx
&=&\int_{2B}|Xv|^{q-2}Xv\cdot X(\min\{v,k\})dx\\
&=&\int_{2B}\min\{v, k\}d\nu\ \leq\ k\ \nu(K)\ .\nonumber
\end{eqnarray}

Consequently,
\begin{equation}\label{kes}
{\rm cap}_{q}(\{v\geq k\},2B)\leq \int_{2B}|X(\min\{v,k\}/k)|^qdx\leq k^{1-q}\nu(K)
\end{equation}
for any $k>0$. Inequality \eqref{kes} with $k=\inf_{B(x,\tfrac{R}{4})}v$
then gives
\begin{eqnarray*}
R^{-q}|B(x,R)|&\leq& C\, {\rm cap}_{q}(B(x,\tfrac{R}{4}), B(x,4R))\\
&\leq& C\, {\rm cap}_{q}(\{v\geq k\},2B)\\
&\leq& C  k^{1-q}\nu(K),
\end{eqnarray*}
which yields the estimate
\begin{equation}\label{esmin}
\inf_{B(x,\tfrac{R}{4})}v\leq C \left(\frac{\nu(K)}{R^{-q}|B(x,R)|}\right)^{\frac{1}{q-1}}.
\end{equation}

Combining \eqref{uwes}, \eqref{wtheterm} and \eqref{esmin} we obtain
\eqref{claim}, thus proving the claim. Note that for any $\varphi\in
C^{\infty}_{0}(2B)$  such that $\varphi\geq \chi_{K}$, by H\"older's
inequality and by applying \eqref{deres} with $k=c$ we have
\begin{eqnarray*}
\nu(K)&\leq& \int_{2B}\varphi d\nu=\int_{\Om}|Xv|^{q-2}Xv\cdot X\varphi dx\\
&\leq& \left(\int_{2B}|Xv|^{q}dx\right)^{\frac{q-1}{q}}\left(\int_{2B}|X\varphi|^{q}dx\right)^{\frac{1}{q}}\\
&\leq&
[c\,\nu(K)]^{\frac{q-1}{q}}\left(\int_{2B}|X\varphi|^{q}dx\right)^{\frac{1}{q}}.
\end{eqnarray*}

Thus minimizing over such functions $\varphi$ we obtain
$$\nu(K)\leq c^{q-1}\, {\rm cap}_{q}(K, 2B)\ .  $$

The latter inequality and \eqref{anular} give
\begin{eqnarray*}
{\rm cap}_{q}((\RR^n\setminus\Om)\cap\overline{B}, 2B)&\geq& {\rm cap}_{q}(K, 2B)\geq C\, \nu(K)=C R^{p-q}\mu(K)\\
&=& C R^{p-q}{\rm cap}_{p}(K, 2B)\geq C R^{p-q}{\rm cap}_{p}(\overline{B}, 2B)\\
&\geq& C R^{-q} |B|\geq C\,  {\rm cap}_{q}(\overline{B}, 2B)
\end{eqnarray*}
 by \eqref{equm} and the uniform $(X, p)$-fatness of $K$.
This proves that $\RR^n\setminus\Om$ is uniformly $(X, q)$-fat, thus
completing the proof of the theorem.

\end{proof}

In what follows given $f\in L^{1}_{\rm loc}(\RR^n)$ we will denote
by $\mathcal{M}_{R}$, $0<R<\infty$, the truncated centered
Hardy-Littlewood maximal function of $f$ defined by
$$\mathcal{M}_{R}(f)(x)=\sup_{0<r\leq R}\frac{1}{|B(x,r)|}\int_{B(x,r)}|f(y)|dy\ ,\quad \quad x\in\RR^n\ .$$

We note explicitly that if $R_1<R_2$, then $\mathcal
M_{R_1}(f)(x)\leq \mathcal M_{R_2}(f)(x)$. The first consequence of
the self-improvement property of uniformly $(X,p)$-fat set is the
following pointwise Hardy inequality which generalizes a result
originally found by Haj\l asz \cite{Ha} in the Euclidean setting.

\begin{theorem}\label{PH} Let $\Om\subset\RR^n$ be a bounded domain with local parameters $C_0$ and $R_0$.
Suppose that $\RR^n\setminus\Om$ is uniformly $(X,p)$-fat with
constants $c_0$ and $r_0$, where $0<r_0\leq R_0/100$ is as in
Theorem \ref{qthick}. There exist $1<q<p$ and a constant $C>0$, both
depending on $C_0$ and $p$, such that the inequality
\begin{equation}\label{pwh1}
|u(x)|\leq C \delta(x)\Big(\mathcal{M}_{4\delta(x)}(|\nabla u|^q)(x)\Big)^{\frac{1}{q}}
\end{equation}
holds for  all $x\in\Om$ with $\delta(x)<r_0$ and  all compactly
supported $u\in C^{0,1}_d(\Om)$.
\end{theorem}

\begin{proof}
For $x\in\Om$ with $\delta(x)<r_0$, we let
$B=B(\overline{x},\delta(x))$, where $\overline{x}\in\partial\Om$ is
chosen so that $|x-\overline{x}|=\delta(x)$. By the fatness
assumption and Theorem \ref{qthick}, there exists $1<q<p$ such that
$${\rm cap}_{1,\,q}(\overline{B}\cap(\RR^n\setminus\Om), 2B)\geq C|B|\delta(x)^{-q}.$$

Thus by Lemma \ref{SobCap} above and Theorem 1.1 in \cite{CDG},
\begin{eqnarray}\label{uofx}
\lefteqn{u(x) \leq |u(x)-u_{B}| +|u_{B}|}\\
&\leq& C \int_{2B}|Xu(y)| \frac{d(x,y)}{|B(x, d(x,y))|}dy+C \Big(\frac{\int_{2B}|Xu|^q dx }{|B|\delta(x)^{-q}}
\Big)^{\frac{1}{q}}.\nonumber
\end{eqnarray}

Note that  by the doubling property \eqref{doubling},
\begin{eqnarray}\label{Max1}
\lefteqn{\int_{2B}|Xu(y)| \frac{d(x,y)}{|B(x, d(x,y))|}dy}\\
&\leq& \int_{B(x,4\delta(x))}|Xu(y)| \frac{d(x,y)}{|B(x,d(x,y))|}dy \nonumber\\
&=& \sum_{k=0}^{\infty} \int_{B(x, 2^{-k}4\delta(x))\setminus
B(x, 2^{-k-1}4\delta(x))}|Xu(y)| \frac{d(x,y)}{|B(x, d(x,y))|}dy\nonumber\\
&\leq& C \sum_{k=0}^{\infty} \frac{2^{-k}4\delta(x)}{|B(x,2^{-k}4\delta(x) )|}\int_{B(x, 2^{-k}4\delta(x))}|Xu(y)| dy\nonumber\\
&\leq& C \delta(x) \mathcal{M}_{4\delta(x)}(|Xu|)(x)\nonumber.
\end{eqnarray}

Also,
\begin{eqnarray}\label{Max2}
\Big(\frac{\int_{2B}|Xu|^q dx }{|B|\delta(x)^{-q}}
\Big)^{\frac{1}{q}}&\leq& C\delta(x) \Big(\frac{\int_{B(x,4\delta(x))}|Xu|^q dx }{|B(x, 4\delta(x))|}
\Big)^{\frac{1}{q}}\\
&\leq&  C \delta(x) \Big(\mathcal{M}_{4\delta(x)}(|Xu|^{q})(x)\Big)^{\frac{1}{q}}.\nonumber
\end{eqnarray}

>From \eqref{uofx}, \eqref{Max1}, \eqref{Max2} and H\"older's
inequality we  now obtain
$$u(x)\leq C\delta(x)\Big(\mathcal{M}_{4\delta(x)}(|Xu|^{q})(x)\Big)^{\frac{1}{q}},$$
which completes the proof of the theorem.
\end{proof}

As it turns out, the pointwise Hardy inequality \eqref{pwh1} is in
fact equivalent to certain geometric conditions on the boundary of
$\Om$ that can be measured in terms of a Hausdorff content. We
introduce the relevant definition.

\begin{definition}\label{sec} Let $s\in \RR$, $r>0$ and $E\subset \RR^n$. The $(X, s, r)$-Hausdorff content of $E$
is the number
$$\mathcal{\widetilde{H}}^{s}_{r}(E)=\inf\sum_{j}r_j^{s}|B_j|\ ,$$
where the infimum is taken over all coverings of $E$ by balls $B_j=B(x_j, r_j)$ such that $x_j\in E$ and $r_j\leq r$.
\end{definition}

We next follow the idea in \cite{Lehr} to prove the following
important consequence of the pointwise Hardy inequality
\eqref{pwh1}.

\begin{theorem}\label{T1}
Let $\Om\subset\RR^n$ be a bounded domain with local parameters
$C_0$ and $R_0$.  Suppose that there exist $r_0\leq R_0/100$, $q>0$
and a constant $C>0$ such that the inequality
\begin{equation}\label{pwh}
|u(x)|\leq C \delta(x)\Big(\mathcal{M}_{4\delta(x)}(|\nabla u|^q)(x)\Big)^{\frac{1}{q}}
\end{equation}
holds for  all $x\in\Om$ with $\delta(x)<r_0$ and  all compactly
supported $u\in C^{0,1}_d(\Om)$. There exists $C_1>0$ such that the
inequality
\begin{equation}\label{Con1}
\mathcal{\widetilde{H}}^{-q}_{\delta(x)}(\overline{B}(x, 2\delta(x))\cap\partial\Om)\geq C_1\delta(x)^{-q}|B(x, \delta(x))|
\end{equation}
holds for all $x\in\Om$ with $\delta(x)<r_0$.
\end{theorem}

\begin{proof} We argue by contradiction and suppose that \eqref{Con1} fails. We can thus find a sequence
$\{x_k\}_{k=1}^{\infty}\subset\Om$, with $\delta(x_k)<r_0$, such
that
\begin{equation*}
\mathcal{\widetilde{H}}^{-q}_{\delta(x)/4}(\overline{B}(x_k,
5\delta(x_k))\cap\partial\Om)< k^{-1}\delta(x_k)^{-q}|B(x_k,
\delta(x_k))|\ .
\end{equation*}

Here, we have used the fact that, by the continuity of the distance
function $\delta$ and the doubling property \eqref{doubling}, the
inequality \eqref{Con1}, which holds for all $x\in\Om$ with
$\delta(x)<r_0$, is equivalent to the validity of
\begin{equation*}
\mathcal{\widetilde{H}}^{-q}_{\delta(x)/4}(\overline{B}(x, 5\delta(x))\cap\partial\Om)\geq C_2\delta(x)^{-q}|B(x, \delta(x))|
\end{equation*}
for all $x\in\Om$ with $\delta(x)<r_0$ and for a constant $C_2>0$.
By compactness, we can now find a finite covering
$\{B_i\}_{i=1}^{N}$, $B_i=B(z_i, r_i)$ with $z_i\in
\overline{B}(x_k, 5\delta(x_k)) \cap\partial\Om$ and
$0<r_i<\delta(x_k)/4$, such that
\begin{equation}\label{overB}
\overline{B}(x_k, 5\delta(x_k))\cap\partial\Om\ \subset\
\bigcup_{i=1}^{N}B_i
\end{equation}
and
\begin{equation}\label{Hleq}
\sum_{i=1}^{N}r_i^{-q}|B_i|<k^{-1}\delta(x_k)^{-q}|B(x_k,
\delta(x_k))|\ .
\end{equation}

Next, for each $k\in \NN$, we define a function $\varphi_k$ by
$$\varphi_{k}(x)=\min\{1,\min_{1\leq i\leq N} r_i^{-1}{\rm dist}(x, 2B_i)\}$$
and let $\phi_{k}\in C_{d}^{0,1}(B(x_{k}, 5\delta(x_k)))$ be such
that $0\leq\phi_{k}\leq 1$ and $\phi_k \equiv 1$ on $B(x_{k},
4\delta(x_k))$. Clearly, the function $u_k=\phi_k\varphi_k$ belongs
to $C^{0,1}_d(\Om)$ and, in view of \eqref{overB}, it has compact
support. Moreover, $u_k(x_k)=1$ since from  the fact that
$z_i\in\partial\Om$  we have
\begin{equation}\label{distxk}
d(x_k,z_i)\geq \delta(x_k)> 4 r_i
\end{equation}
for all $1\leq i\leq N$. Also, since $\varphi_{k}(x)=1 $ for
$x\not\in\bigcup_{i=1}^{N}3\overline{B}_{i}$ and $\varphi_{k}(x)=0 $
for $x\in\bigcup_{i=1}^{N}2\overline{B}_{i}$, it is  easy to see that
$$supp\ (|Xu_{k}|)\cap B(x_k, 4 \delta(x_k))\subset \bigcup_{i=1}^{N}(3\overline{B}_{i}\setminus 2B_i)$$
and that for a.e. $y\in B(x_k, 4\delta(x_k))$ we have
\begin{equation}\label{nabu}
|Xu_{k}(y)|^{q}\leq
\sum_{i=1}^{N}r_i^{-q}\chi_{3\overline{B}_i\setminus 2B_i}(y)\ .
\end{equation}

Hence using \eqref{distxk} and \eqref{nabu} we can calculate

\begin{eqnarray}\label{maxf}
\lefteqn{\mathcal{M}_{4\delta(x_k)}(|Xu_{k}|^{q})(x_k)}\\
&\leq& C \sup_{\frac{1}{4}\delta(x_k)\leq r\leq 4\delta(x_k)}
\frac{1}{|B(x_k, r)|}\int_{B(x_k, r)}|X u_{k}(y)|^{q}dy\nonumber\\
&\leq& C \frac{1}{|B(x_k, \delta(x_k))|}\int_{B(x_k, 4\delta(x_k))}|X u_{k}(y)|^{q}dy\nonumber\\
&\leq& C \frac{1}{|B(x_k, \delta(x_k))|}\sum_{i=1}^{N}|3\overline{B_i}\setminus 2B_i| r_i^{-q}\nonumber\\
&\leq& C \frac{1}{|B(x_k, \delta(x_k))|}\sum_{i=1}^{N}|B_i| r_i^{-q}.\nonumber
\end{eqnarray}

From \eqref{Hleq} and \eqref{maxf} we obtain
$$\delta(x_k)^{q}\mathcal{M}_{4\delta(x_k)}(|Xu_{k}|^{q})(x_k)\leq C k^{-1}.$$
Since $u_{k}=1$ for any $k$, this implies that the pointwise Hardy
inequality \eqref{pwh} fails to hold with a uniform constant for all
compactly supported $u\in C^{0,1}_d(\Om)$. This contradiction
completes the proof of the theorem.

\end{proof}

As in \cite{Lehr}, from \eqref{Con1} we can also obtain the following thickness condition on $\RR^n\setminus\Om$.

\begin{theorem}\label{T2}
Let $\Om\subset\RR^n$ be a bounded domain with local parameters
$C_0$ and $R_0$.  Suppose that there exist $r_0\leq R_0/100$, $q>0$
and a constant $C>0$ such that the inequality
\begin{equation}\label{Con12}
\mathcal{\widetilde{H}}^{-q}_{\delta(x)}(\overline{B}(x, 2\delta(x))\cap\partial\Om)\geq C\delta(x)^{-q}|B(x, \delta(x))|
\end{equation}
holds for all $x\in\Om$ with $\delta(x)<r_0$. Then, there exists
$C_1>0$ such that
\begin{equation}\label{Hr}
\mathcal{\widetilde{H}}^{-q}_{r}(B(w,r)\cap (\RR^n\setminus\Om))\geq C_1 r^{-q}|B(w,r)|
\end{equation}
for all $w\in\partial\Om$ and $0<r<r_0$.
 \end{theorem}

\begin{proof}
Let $w\in\partial\Om$ and $0<r<r_0$. If
\begin{equation*}
|B(w, \tfrac{r}{2})\cap(\RR^n\setminus\Om)|\geq \tfrac{1}{2}|B(w,\tfrac{r}{2})|
\end{equation*}
then it is easy to see that \eqref{Hr} holds with $C_1=2^{-Q}C_0/2$.
Thus we may assume that
\begin{equation*}
|B(w, \tfrac{r}{2})\cap\Om|\geq \tfrac{1}{2}|B(w,\tfrac{r}{2})|\ ,
\end{equation*}
which by \eqref{doubling} gives
\begin{equation}\label{1minus}
|B(w, \tfrac{r}{2})\cap\Om|\ \geq\ 2^{-Q}C_0\ |B(w,r)|/2\ .
\end{equation}

Now to prove \eqref{Hr} it is enough to show that
\begin{equation}\label{Hr1}
\mathcal{\widetilde{H}}^{-q}_{r}(B(w,r)\cap \partial\Om)\geq C_1
r^{-q}|B(w,r)|\ .
\end{equation}

To this end, let $\{B_i\}_{i=1}^{\infty}$, $B_i=B(z_i, r_i)$ with
$z_i\in\partial\Om$ and $0<r_i\leq r$ be a covering of
$B(w,r)\cap\partial\Om$. Then if
$$\sum_{i}|B_i|\geq  (2^{-Q}C_0)^{2}|B(w,r)|/4,$$
it follows that \eqref{Hr1} holds with
$C_1=\frac{1}{4}(2^{-Q}C_0)^{2}$. Hence, we are left with
considering only the case
\begin{equation}\label{sumbi}
\sum_{i}|B_i|<(2^{-Q}C_0)^{2}|B(w,r)|/4\ .
\end{equation}

Using  \eqref{doubling}, \eqref{1minus} and \eqref{sumbi} we can  now  estimate
\begin{eqnarray*}
|(B(w,\tfrac{r}{2})\cap\Om)\setminus \bigcup_{i}2B_i|
&\geq& |B(w,\tfrac{r}{2})\cap\Om|-  2^{Q}C_0^{-1}\sum_i |B_i|\\
&\geq& 2^{-Q}C_0|B(w,r)|/2- 2^{-Q}C_0|B(w,r)|/4\\
&=&2^{-Q}C_0|B(w,r)|/4\ .
\end{eqnarray*}

Thus by a covering lemma (see \cite{St1}, page 9) we can find a
sequence of pairwise disjoint balls $B(x_k,6\delta(x_k))$ with
$x_k\in (B(w,\tfrac{r}{2})\cap\Om)\setminus \bigcup_{i}2B_i$ such that
\begin{equation*}
|B(w,r)|\leq C|(B(w,\tfrac{r}{2})\cap\Om)\setminus
\bigcup_{i}2B_i|\leq C \sum_k |B(x_k,30\delta(x_k))|\ .
\end{equation*}

This together with \eqref{doubling} and \eqref{Con12} give
\begin{eqnarray}\label{BA}
|B(w,r)|r^{-q}&\leq& C \sum_k |B(x_k,\delta(x_k))|\delta(x_k)^{-q}\\
&\leq& C \sum_{k}\mathcal{\widetilde{H}}^{-q}_{\delta(x_k)}(\overline{B}(x_k,2\delta(x_k))\cap\partial\Om)\nonumber
\end{eqnarray}
since $\delta(x_k)<\tfrac{r}{2}$ for all $k$.

We next observe that we can further  assume that
\begin{equation}\label{delta}
\delta(x)<\tfrac{r}{4} {\rm ~for~ all~ } x\in
B(w,\tfrac{r}{2})\cap\Om\ .
\end{equation}

In fact, if there exits $x\in B(w,\tfrac{r}{2})\cap\Om$ such that
$\delta(x)\geq  \tfrac{r}{4}$, then there exists $x_0\in
B(w,\tfrac{r}{2})\cap\Om$ such that  $\delta(x_0)= \tfrac{r}{4}$ by
the continuity of $\delta$. Thus $B(x_0, 2\delta(x_0))\subset
B(w,r)$, and in view of assumption \eqref{Con12} we obtain
\begin{eqnarray*}
\mathcal{\widetilde{H}}_{r}^{-q}(B(w,r)\cap\partial\Om)&\geq& C \mathcal{\widetilde{H}}_{\delta(x_0)}
^{-q}(B(x_0,2\delta(x_0)\cap\partial\Om))\\
&\geq& C \delta(x_0)^{-q}|B(x_0, \delta(x_0))|\geq C r^{-q}|B(w,r)|\
,
\end{eqnarray*}
which gives \eqref{Hr1}. Now, inequality \eqref{delta} in particular
implies that
$$\overline{B}(x_k,2\delta(x_k))\cap\partial\Om\subset B(w,r)\cap\partial\Om\subset \bigcup_i B_i\ ,$$ and
hence for every $k$ one has
\begin{equation}\label{H}
\mathcal{\widetilde{H}}_{2\delta(x_k)}^{-q}(\overline{B}(x_k,2\delta(x_k)\cap\partial\Om))
\leq \sum_{\{i\in \mathbb N\mid B_i\cap
\overline{B}(x_k,2\delta(x_k))\not=\emptyset\}} |B_i|r_i^{-q}\ .
\end{equation}

Here we have used the fact that  $r_i <2\delta(x_k)$ since
$x_k\not\in 2B_i$. From \eqref{BA} and \eqref{H}, after changing the
order of summation, we obtain
\begin{eqnarray}\label{Bleqsum}
|B(w,r)|r^{-q}&\leq& C \sum_{i}\sum_{\{k\in \mathbb N\mid B_i\cap \overline{B}(x_k,2\delta(x_k))\not=\emptyset\}} |B_i|r_i^{-q}\\
&\leq& C \sum_i C(i)|B_i|r_i^{-q}\ ,\nonumber
\end{eqnarray}
where $C(i)$ is the number of balls $\overline{B}(x_k,2\delta(x_k))$
that intersect $B_i$. Note that if $B_i\cap
\overline{B}(x_k,2\delta(x_k))\not=\emptyset$, then since  $r_i <
2\delta(x_k)$ we see that $B_i\subset B(x_k,6\delta(x_k))$. Hence
$C(i)\leq 1$ for all $i$ since by our choice the balls $B(x_k,
6\delta(x_k))$ are pairwise disjoint.  This and \eqref{Bleqsum} give
$$|B(w,r)|r^{-q}\leq C \sum_i |B_i|r_i^{-q}$$
and inequality \eqref{Hr1} follows  as the coverings $\{B_i\}_i$ of
$B(w,r)\cap\partial\Om$ are arbitrary. This completes the proof of
the theorem.

\end{proof}

The thickness condition \eqref{Hr} that involves the Hausdorff
content will now be shown to imply the uniform $(X,p)$-fatness of
$\RR^n\setminus\Om$. To achieve this we borrow an idea from
\cite{HK}.

\begin{theorem}\label{T3}
Let $\Om\subset\RR^n$ be a bounded domain with local parameters
$C_0$ and $R_0$.  Suppose that there exist $r_0\leq R_0/100$, $1<q<
p$ and a constant $C>0$ such that the inequality
\begin{equation}\label{Hr2}
\mathcal{\widetilde{H}}^{-q}_{r}(B(w,r)\cap (\RR^n\setminus\Om))\geq C r^{-q}|B(w,r)|
\end{equation}
holds for all $w\in\partial\Om$ and $0<r<r_0$. Then, there exists
$C_1>0$ such that the  $\RR^n\setminus\Om$ is uniformly $(X,p)$-fat
with constants $C_1$ and $r_0$.
\end{theorem}
\begin{proof}
Let $z\in \partial\Om$ and $0<r<r_0$. We need to find a constant $C_1>0$ independent of $z$ and $r$ such that
\begin{equation}\label{cappK}
{\rm cap}_{p}(K, B(z,2r))\geq C_1 r^{-p}|B(z,r)|\ ,
\end{equation}
where $K= (\RR^n\setminus\Om)\cap \overline{B}(z,r)$. From
\eqref{Hr2}  we have
\begin{equation}\label{Hr3}
\widetilde{\mathcal H}^{-q}_{r}(K)\geq C r^{-q}|B(z,r)|\ .
\end{equation}

Let $\varphi\in C_0^\infty(B(z,2r))$ be such that $\varphi\geq 1$ on $K$. If there is $x_0\in K$ such that
$$|\varphi(x_0)-\varphi_{B(x_0,4r)}|\leq 1/2\ ,$$ then
$$1\leq\varphi(x_0)\leq |\varphi(x_0)-\varphi_{B(x_0,4r)}|+|\varphi_{B(x_0,4r)}|\leq 1/2+ |\varphi_{B(x_0,4r)}|\ .$$

By Lemma \ref{SobCap}, the  doubling property \eqref{doubling} and
\eqref{anular} we obtain
$$1/2\leq |\varphi_{B(x_0,4r)}|\leq C \Big(r^{p}|B(z,r)|^{-1}\int_{B(z,2r)}|X\varphi|^{p}dx\Big)^{\frac{1}{p}}\ ,$$
which gives \eqref{cappK}. Thus we may assume that
\begin{equation*}
1/2 <|\varphi(x)-\varphi_{B(x,4r)}| \quad {\rm for~ all~} x\in K.
\end{equation*}

Under such assumption, using the covering argument in Theorem 5.9 in
\cite{HK}, the inequality \eqref{cappK} follows from \eqref{Hr3} and
from Theorem \ref{T:jer}.

\end{proof}

Finally, we summarize in one single theorem the results obtained in
Theorems \ref{PH}, \ref{T1}, \ref{T2} and \ref{T3}.

\begin{theorem}\label{summa}
Let $\Om\subset\RR^n$ be a bounded domain with local parameters
$C_0$ and $R_0$ and let $1<p<\infty$. There exists $0<r_0\leq
R_0/100$ such that  the following statements are equivalent:

{\rm (i)} The set $\RR^n\setminus\Om$ is uniformly $(X,p)$-fat with
constants $c_0$ and $r_0$ for some $c_0>0$. That is,
\begin{equation*}
{\rm cap}_{p}((\RR^n\setminus\Om)\cap \overline{B}(w,r),
B(w,2r))\geq c_0 r^{-p}|B(w,r)|
\end{equation*}
for all $w\in \partial\Om $ and $0<r<r_0$.

{\rm (ii)} There  exist  $1< q<p $ and a constant $C>0$ such that
\begin{equation*}
|u(x)|\leq C \delta(x)\Big(\mathcal{M}_{4\delta(x)}(|\nabla
u|^q)(x)\Big)^{\frac{1}{q}}
\end{equation*}
 for all $x\in\Om$ with $\delta(x)<r_0$ and all compactly supported $u\in C^{0,1}_d(\Om)$.

{\rm (iii)} There  exist  $1< q<p $ and  a constant $C>0$ such that
\begin{equation*}
\mathcal{\widetilde{H}}^{-q}_{\delta(x)}(\overline{B}(x,
2\delta(x))\cap\partial\Om)\geq C\delta(x)^{-q}|B(x, \delta(x))|
\end{equation*}
 for all $x\in\Om$ with $\delta(x)<r_0$.

{\rm(iv)} There  exist  $1< q<p $ and  a constant $C>0$ such that
\begin{equation*}
\mathcal{\widetilde{H}}^{-q}_{r}(B(w,r)\cap (\RR^n\setminus\Om))\geq
C r^{-q}|B(w,r)|
\end{equation*}
for all $w\in\partial\Om$ and $0<r<r_0$.
\end{theorem}

\begin{remark}\label{R:lehr}
As an example in \cite{Lehr} shows, we cannot replace the set
$\Rn\setminus \Om$ in the statement $\rm(iv)$ in Theorem \ref{summa}
with the smaller set $\p \Om$.
\end{remark}

\section{ Hardy Inequalities on bounded domains}\label{S:hbd}

Our first result in this section is the following Hardy inequality which is a consequence of Theorem \ref{PH} and the
$L^s$ boundedness of the Hardy-Littlewood maximal function for $s>1$. We remark that no assumption on the
smallness of the diameter of the domain is required  as opposed to  Poicar\'e's inequality \eqref{poincare} and Sobolev's
inequalities established in \cite{GN1}.

\begin{theorem}\label{HI}Let $\Om\subset\RR^n$ be a bounded domain with local parameters $C_0$ and $R_0$.
Suppose that $\RR^n\setminus \Om$ is uniformly $(X, p)$-fat with
constants $c_0>0$ and $0<r_0\leq R_0/100$. There is a constant $C>0$
such that for all $\varphi\in C^{\infty}_{0}(\Om)$
\begin{equation}\label{hardy}
\int_{\Om}\frac{|\varphi(x)|^{p}}{\delta(x)^{p}}\ dx\ \leq\ C\
\int_{\Om}|X\varphi|^{p}\ dx\ .
\end{equation}
\end{theorem}

\begin{proof}
Let $\Om_{r_0}=\{ x\in\Om: \delta(x)\geq r_0\}$ and let  $\varphi\in C^{\infty}_{0}(\Om)$.
By Theorem \ref{PH} we can find $1<q<p$ such that
\begin{eqnarray*}
\int_{\Om}|\varphi(x)|^{p}\delta(x)^{-p}dx
&=& \int_{\Om_{r_0}}|\varphi(x)|^{p}\delta(x)^{-p}dx+
\int_{\Om\setminus\Om_{r_0}}|\varphi(x)|^{p}\delta(x)^{-p}dx\\
&\leq& r_{0}^{-p}\int_{\Om}|\varphi(x)|^{p}dx+
C \int_{\Om}\Big(\mathcal{M}_{4r_0}(|X\varphi|^{q})(x)\Big)^{\frac{p}{q}}dx\\
&\leq& C \int_{\Om}|X\varphi(x)|^{p}dx\ .
\end{eqnarray*}

In the last inequality above we have used the Poincar\'e inequality
\eqref{poi} and the boundedness property of $\mathcal{M}_{4r_0}$ on
$L^{s}(\Om)$, $s>1$, (see \cite{St1}). The proof of Theorem \ref{HI}
is then complete.

\end{proof}

To state Theorems \ref{GP} and \ref{FPT}  below, we need to fix a
Whitney decomposition of $\Om$ into balls as in the following lemma,
whose construction can be found for example in  \cite{J} or
\cite{FS}.

\begin{lemma}\label{Whitney} Let $\Om\subset\RR^n$ be a bounded domain with local parameters $C_0$ and $R_0$. There exists a family of balls ${\mathcal W}=\{B_{j}\}$ with
$B_j=B(x_j,r_j)$ and a constant $M>0$   such that
\begin{eqnarray*}
&{\rm (a)}& \Om\subset \cup_{j}B_{j},\\
&{\rm (b)}& B(x_j,\tfrac{r_j}{4})\cap B(x_k,\tfrac{r_k}{4})\not=\varnothing\quad {\rm for~} j\not=k, \\
&{\rm (c)}&  r_j=10^{-3}\min\{R_0/{\rm diam}(\Om),1\} {\rm dist}(B_{j},\partial\Om),\\
&{\rm (d)}& \sum_{j}\chi_{4B_j}(x)\leq M \chi_{\Om}(x).
\end{eqnarray*}
In ${\rm (c)}$,  $${\rm diam}(\Om)=\sup_{x,y\in\Om}d(x,y)$$ is the
diameter of $\Om$ with respect to the CC metric. In particular we have $r_j\leq 10^{-3}R_0$.
\end{lemma}

We can now go further in characterizing weight functions $V$ on
$\Om$ for which the embedding
\begin{equation*}
\int_{\Om}|\varphi(x)|^{p}\,V(x)dx\leq C\int_{\Om}|X\varphi|^{p}dx
\end{equation*}
holds for all   $\varphi\in C^{\infty}_{0}(\Om)$. Here the condition on $V$ is formulated in terms of a
localized capacitary condition adapted to a Whitney decomposition of $\Om$. Such a condition
can be simplified further  in the setting of Carnot groups as we point out in Remark \ref{HG} below.
In the Euclidean setting it was used in \cite{HMV} to characterize the solvability of multi-dimensional Riccati
equations on bounded domains.

\begin{theorem}\label{GP}Let $\Om\subset\RR^n$ be a bounded domain with local parameters $C_0$ and $R_0$.
Let $V\geq 0$ be in $L^{1}_{\rm loc}(\Om)$. Suppose that  $\RR^n\setminus \Om$ is uniformly $(X,p)$-fat with
$1<p<Q$.  Then the embedding
\begin{equation}\label{trace}
\int_{\Om}|\varphi(x)|^{p}\,V(x)dx\leq C\int_{\Om}|X\varphi|^{p}dx,\qquad \varphi\in C^{\infty}_{0}(\Om),
\end{equation}
holds if and only if
\begin{equation}\label{capcond}
\sup_{B\in\mathcal{W}}\sup_{\substack{ K \subset 2B\\ K\, {\rm compact}}}
\frac{\int_{K}V(x)dx}{{\rm cap}_{p}(K,\Om)}\leq C,
\end{equation}
where $\mathcal{W}=\{B_j\}$ is a Whitney decomposition of $\Om$ as
in Lemma \ref{Whitney}.
\end{theorem}

\begin{remark} \label{HG}  In the setting of a
Carnot group $\bf G$ with homogeneous dimension $Q$, we can replace
${\rm cap}_{p}(K,\Om)$ by ${\rm cap}_{p}(K,\bf G)$ in
$\eqref{capcond}$ since if $B\in\mathcal{W}$ and $K$ is a compact
set in $2B$ we have
\begin{equation}\label{cpcapac}
c\ {\rm cap}_{p}(K,\Om)\ \leq {\rm cap}_{p}(K,{\bf G})\ \leq\  {\rm cap}_{p}(K,\Om)\ .
\end{equation}
The second inequality in \eqref{cpcapac} is obvious. To see the first one, let $\varphi\in C^{\infty}_{0}({\bf G})$,
$\varphi\geq 1$ on $K$,
and choose a cut-off function $\eta\in C^{\infty}_{0}(4B)$ such that $0\leq \eta\leq 1$, $\eta\equiv 1$ on $2B$ and
$|X\eta|\leq \frac{C}{r_B}$, where $r_B$ is the radius of $B$.  Since $\varphi\eta\in C^{\infty}_{0}(\Om)$,
$\varphi\eta\geq 1$ on $K$, we have
\begin{eqnarray*}
{\rm cap}_{p}(K,\Om)&\leq& \int_{\Om} |X (\varphi\eta)|^{p}dg\\
&\leq& \int_{{\bf G}} |X\varphi|^{p}dg + C\int_{{4B\setminus 2B}} \frac{|\varphi|^{p}}{r_B^{p}}dg\\
&\leq& \int_{{\bf G}} |X\varphi|^{p}dg + C\int_{{\bf G}}
\frac{|\varphi|^{p}}{\rho(g,g_0)^p}dg,
\end{eqnarray*}
where $g_0$ is the center of $B$, and we have denoted by
$\rho(g,g_0)$ the pseudo-distance induced  on $\bf G$ by the
anisotropic Folland-Stein gauge, see \cite{FS}, \cite{F2}. To bound
the third integral in the right-hand side of the latter inequality
we use the following Hardy type inequality
\begin{equation}\label{pHG}
\int_{\bf G}\frac{\varphi^{p}}{\rho(g,g_0)^{p}}\ dg\ \leq\ C\
\int_{\bf G} |X\varphi|^{p}\ dg, \qquad \varphi\in
C^{\infty}_{0}(\bf G),
\end{equation}
which is easily proved as follows. Recall the Folland-Stein Sobolev
embedding, see \cite{F2}, \begin{equation}\label{fs} \left(\int_{\bf
G} |\varphi|^{\frac{pQ}{Q-p}}\ dg\right)^{\frac{Q-p}{pQ}} \leq S_{p}
\left(\int_{\bf G} |X\varphi|^p \ dg\right)^{\frac{1}{p}}\ ,\ \ \
\varphi\in C^\infty_0(\bf G)\ .
\end{equation}
Observing that for every $g_0\in \bf G$ one has  $g\to
\frac{1}{\rho(g,g_0)^{p}}\in L^{Q/p,\infty}(\bf G)$, from the
generalized H\"older's inequality for weak $L^{p}$ spaces due to R.
Hunt \cite{Hu} one obtains with an absolute constant $B>0$
\begin{align*}
\int_{\bf G}\frac{\varphi^{p}}{\rho(g,g_0)^{p}}\ dg\ & \leq\ B
\left(\int_{\bf G} |\varphi|^{\frac{pQ}{Q-p}}\
dg\right)^{\frac{Q-p}{Q}}
||\rho(\cdot,g_0)^{-p}||_{L^{Q/p,\infty}(\bf G)}
\\
& \leq\ C\ \int_{\bf G} |X\varphi|^{p}\ dg \ ,
\end{align*}
where in the last inequality we have used \eqref{fs}. This proves
\eqref{pHG}. In conclusion we find
$${\rm cap}_{p}(K,\Om)\leq C \int_{{\bf G}} |X \varphi|^{p}dg,$$
which gives the first inequality in \eqref{cpcapac}.

\end{remark}

\begin{proof}[Proof of Theorem \ref{GP}]
That the emdedding \eqref{trace} implies the capacitary condition \eqref{capcond} is clear.
To prove the converse, we let $\{\phi_{j}\}$ be a Lipschitz  partition of unity
associated with the Whitney decomposition $\mathcal{W}=\{B_j\}$ (see \cite{GN2}). That is,
$0\leq \phi_j\leq 1$ is Lipschitz with respect to the CC metric, $ supp \ \phi_j\Subset 2B_j$,
$|X\phi_{j}|\leq C/{\rm diam}(B_j)$, and $$\sum_{j}\phi_j(x)=\chi_{\Om}(x).$$
Moreover, by property (d) in Lemma \ref{Whitney}, there is a constant $C(p)$ such that
$$\left(\sum_{j}\phi_j(x)\right)^{p}=C(p)\sum_{j}\phi_j(x)^p.$$
Then for any $\varphi\in C^{\infty}_{0}(\Om)$, we have
\begin{eqnarray*}
\int_{\Om}|\varphi(x)|^{p}\,V(x)dx &\leq& C\ \sum_{j}\int_{\Om}|\phi_j\varphi(x)|^{p}\,V(x)dx \\
&\leq& C\ \sum_{j}\int_{4B_j}|X(\phi_j\varphi)|^{p}dx
\end{eqnarray*}
by \eqref{capcond} and Theorem 5.3 in \cite{D2}. Thus from Theorem \ref{HI} and Lemma \ref{Whitney}, we obtain
\begin{eqnarray*}
\lefteqn{\int_{\Om}|\varphi(x)|^{p}\,V(x)dx}\\
&\leq& C\ \sum_{j}\int_{4B_j}|X\varphi|^{p}dx + C\sum_{j}[{\rm diam}(B_j)]^{-p}\int_{4B_j}|\varphi|^p dx\\
&\leq& C \int_{\Om}|X\varphi|^{p}dx + C\int_{\Om}|\varphi|^p \delta^{-p}(x) dx\\
&\leq& C \int_{\Om}|X\varphi|^{p}dx.
\end{eqnarray*}
This completes the proof of the theorem.

\end{proof}

In view of Theorem 1.6 in \cite{D2}, the above proof also gives the following Fefferman-Phong  type sufficiency result \cite{Fef}.

\begin{theorem}\label{FPT}Let $\Om\subset\RR^n$ be a bounded domain with local parameters $C_0$ and $R_0$.
Let $V\geq 0$ be in $L^{1}_{\rm loc}(\Om)$. Suppose that  $\RR^n\setminus \Om$ is uniformly $(X, p)$-fat with
$1<p<Q$.  Then the embedding
\begin{equation}\label{trace2}
\int_{\Om}|\varphi(x)|^{p}\,V(x)dx\leq C\int_{\Om}|X\varphi|^{p}dx,\qquad \varphi\in C^{\infty}_{0}(\Om),
\end{equation}
holds if, for some $s>1$, $V$ satisfies the following localized
Fefferman-Phong type condition adapted to $\Om$:
\begin{equation}\label{FP}
\sup_{B\in\mathcal{W}}\sup_{\substack{x\in 2B\\0<r<{\rm diam}(B)}}
\int_{B(x,r)}V(y)^{s}dy\ \leq\ C\ \frac{|B(x,r)|}{r^{sp}}
\end{equation}
 where $\mathcal{W}=\{B_j\}$ is a Whitney
decomposition of $\Om$ as in Lemma \ref{Whitney}.
\end{theorem}

 Let $L^{s,\infty}(\Om)$, $0<s<\infty$, denote  the weak $L^s$ space on $\Om$, i.e.,
$$L^{s,\infty}(\Om)=\left\{ f : \norm{f}_{L^{s,\infty}(\Om)}<\infty \right\},$$
where
$$\norm{f}_{L^{s,\infty}(\Om)}\ =\ \sup_{t>0}\ t\ |\{x\in\Om: |f(x)|>t\}|^{\frac{1}{s}}\ .$$

Equivalently, one can take
$$\norm{f}_{L^{s,\infty}(\Om)}=\sup_{E\subset\Om:\, |E|>0}|E|^{\frac{1}{s}-\frac{1}{r}}\left(\int_{E}|f|^{r}dx\right)^{\frac{1}{r}}$$
for any $0<r<s$. For $s=\infty$, we define
$$L^{\infty,\infty}(\Om)= L^{\infty}(\Om)\ .$$

From Theorem \ref{FP} we obtain the following corollary, which
improves a similar result in \cite{DPT}, Remark 3.7 in the sense
that not only does it cover the subelliptic case but also require a
milder assumption on the boundary.

\begin{corollary} \label{LQ}
Let $\Om\subset\RR^n$ be a bounded domain with local parameters
$C_0$ and $R_0$. Suppose that $\RR^n\setminus\Om$ is uniformly $(X,
p)$-fat for $1<p<Q$, where $Q$ is the homogeneous dimension of
$\Om$. If $v\in L^{\frac{Q}{\gamma},\infty}(\Om)$ for some $0\leq
\gamma\leq p$, then the embedding \eqref{trace2} holds for the
weight $V(x)=\delta(x)^{-p+\gamma}v(x)$.
\end{corollary}
\begin{proof}
Let $\mathcal{W}=\{B_j\}$ is a Whitney decompositon of $\Om$ as in Lemma \ref{Whitney}.
For $x\in 2B$, $B\in\mathcal{W}$, $0<r<{\rm diam}(B)$, and $1<s <\frac{Q}{\gamma}$, we have
$$\int_{B(x,r)}V(y)^{s}dy\leq C r^{-sp+s\gamma}\int_{B(x,r)}v(y)^sdy.$$

It is then easily seen from H\"older's inequality and the doubling
property \eqref{doubling} that
\begin{eqnarray*}
\int_{B(x,r)}V(y)^{s}dy &\leq& C r^{-sp}|B(x,r)|\norm{v}^{s}_{L^{\frac{Q}{\gamma},\infty}(\Om)}
\left(\frac{r}{|B(x,r)|^{\frac{1}{Q}}}\right)^{s\gamma}\\
&\leq& C r^{-sp}|B(x,r)| \norm{v}_{L^{\frac{Q}{\gamma},\infty}(\Om)}.
\end{eqnarray*}
By Theorem \ref{FPT} we obtain the corollary.

\end{proof}

The results obtained in Corollary \ref{LQ} do not in general  cover
the case in which $v(x)$ has a point singularity in $\Om$, such as
$V(x)=\delta(x)^{-p+\gamma} d(x, x_0)^{-\gamma}$, with $0\leq
\gamma\leq p$  and $1<p<Q(x_0)$ for some $x_0\in\Om$, where $Q(x_0)$
is the homogeneous dimension at $x_0$. The reason is that it may
happen that $Q(x_0)<Q$ and hence $d(\cdot,x_0)^{-\gamma} \not\in
L^{\frac{Q}{\gamma},\infty}(\Om)$. However, by the upper estimate in \eqref{lambdarescale}
we still can obtain  inequality \eqref{trace2} for such weights as follows.

\begin{corollary}\label{LQ1}
Let $\Om\subset\RR^n$ be a bounded domain with local parameters
$C_0$ and $R_0$. Given $x_0\in \Om$ suppose that $\RR^n\setminus\Om$
is uniformly $(X, p)$-fat for $1<p<Q(x_0)$. Then for any $0\leq
\gamma\leq p$ the embedding \eqref{trace2} holds for the weight
$$V(x)=\delta(x)^{-p+\gamma} d(x, x_0)^{-\gamma}\ .$$
\end{corollary}

\begin{proof}
Let $\mathcal{W}=\{B_j\}$ be a Whitney decomposition of $\Om$ as in
Lemma \ref{Whitney}. For $x\in 2B$, $B\in\mathcal{W}$, $0<r<{\rm
diam}(B)$, and $1<s <\frac{Q(x_0)}{\gamma}$, we have
\begin{equation}\label{Vs}
\int_{B(x,r)}V(y)^{s}dy\leq C\
r^{-sp+s\gamma}\int_{B(x,r)}d(y,x_0)^{-\gamma s}dy\ .
\end{equation}

Thus if $x\not\in B(x_0,2r)$ then
$$\int_{B(x,r)}V(y)^{s}dy\leq C\ \frac{|B(x,r)|}{r^{sp}}$$
since for such $x$ we have $d(y,x_0)\geq r$ for every $y\in B(x,r)$. On the other hand, if $x\in B(x_0,2r)$
then from \eqref{Vs} we find
\begin{eqnarray*}
\int_{B(x,r)}V(y)^{s}dy&\leq& C\  r^{s\gamma-sp}
\int_{B(x_0,3r)}d(y,x_0)^{-\gamma s}dy\\
&=& C\ r^{s\gamma-sp}\sum_{k=0}^{\infty}\int_{\frac{3r}{2^{k+1}}\leq d(y,x_0)
< \frac{3r}{2^{k}}}d(y,x_0)^{-\gamma s}dy\\
&\leq& C\ r^{s\gamma-sp}\sum_{k=0}^{\infty} \Big(\frac{r}{2^{k}}\Big)^{-\gamma s}
\left|B(x_0, \tfrac{3r}{2^k})\right|.
\end{eqnarray*}

Thus in view of \eqref{lambdarescale} and the doubling property \eqref{doubling} we obtain
\begin{eqnarray*}
\int_{B(x,r)}V(y)^{s}dy&\leq&  C\ r^{s\gamma-sp}\sum_{k=0}^{\infty}
\Big(\frac{r}{2^{k}}\Big)^{-\gamma s}\Big(\frac{1}{2^k}\Big)^{Q(x_0)}|B(x_0,3r)| \\
&\leq& C\ \frac{|B(x_0,3r)|}{r^{sp}} \sum_{k=0}^{\infty}\Big(\frac{1}{2^k}\Big)^{Q(x_0)-\gamma s}\\
&\leq& C(x_0)\ \frac{|B(x,r)|}{r^{sp}}.
\end{eqnarray*}

Thus  by Theorem \ref{FPT} we obtain the corollary.

\end{proof}

\begin{remark}
If we have $\gamma = p$ in Corollary \ref{LQ1}, then  we do not need to assume $\Rn \setminus \Om$ to be
uniformly $(X,p)$-fat. In fact, to obtain the embedding \eqref{trace2} in this case
we  use Theorem 1.6 in \cite{D2}, the Poincar\'e inequality \eqref{poi}, and a finite partition of unity for $\Om$.
\end{remark}

\section{Hardy inequalities with sharp constants}\label{S:sc}

In this section we collect, without proofs, for illustrative
purposes some theorems from the forthcoming article \cite{DGP}. The
relevant results pertain certain Hardy-Sobolev inequalities on
bounded and unbounded domains with a point singularity which are
connected to the results in section \ref{S:hbd}, but are not
included in them.

We begin by recalling that when $X = \{X_1,...,X_m\}$ constitutes an
orthonormal basis of bracket generating vector fields in a Carnot
group $\bf G$, then a fundamental solution $\gp$ for $- \mathcal
L_p$ in all of $\bf G$ was constructed in \cite{DG}. For any bounded
open set $\Om\subset \Rn$ one can construct a positive fundamental
solution with generalized zero boundary values, i.e., a Green
function, in the more general situation of a Carnot-Carath\'eodory
space. Henceforth, for a fixed $x\in \Om$ we will denote by
$\gp(x,\cdot)$ such fundamental solution with singularity at some
fixed $x\in\Om$. This means that $\gp(x,\cdot)$ satisfies the
equation
\begin{equation}\label{wsH}
\int_{\Om } |X\gp(x,y)|^{p-2} <X\gp(x,y),X \varphi(y)> dy\ =\
\varphi(x)\ ,
\end{equation}
for every $\varphi \in C^\infty_0(\Om)$.

We recall the following fundamental estimate, which is Theorem 7.2
in \cite{CDG3}: Let $K\subset \Om\subset \Rn$ be a compact set, with
local parameters $C_0$ and $R_0$. Given $x\in K$, and $1<p<Q(x)$,
there exists a positive constant $C$, depending on $C_0$ and $p$,
such that for any $0<r\leq R_0/2$, and $y\in B(x,r)$ one has
\begin{equation}\label{me}
C\ \left(\frac{d(x,y)^p}{\Lambda(x,d(x,y))}\right)^{\frac{1}{p-1}}\
\leq\ \gp(x,y)\ \leq\ C^{-1}\
\left(\frac{d(x,y)^p}{\Lambda(x,d(x,y))}\right)^{\frac{1}{p-1}}\ .
\end{equation}

The estimate \eqref{me} generalizes that obtained by Nagel, Stein
and Wainger \cite{NSW}, and independently by Sanchez-Calle \cite{SC}
in the case $p = 2$.

For any given $x\in K$, we fix a number $p = p(x)$ such that
$1<p<Q(x)$, and introduce the function
\begin{equation}\label{E}
E(x,r)\ \overset{def}{=}\
\left(\frac{\Lambda(x,r)}{r^p}\right)^{\frac{1}{p-1}}\ .
\end{equation}

Because of the constraint imposed on $p = p(x)$, we see that, for
every fixed $x\in K$, the function $r \to E(x,r)$ is strictly
increasing, and thereby invertible. We denote by $F(x,\cdot) =
E(x,\cdot)^{-1}$, the inverse function of $E(x,\cdot)$, so that
\[
F(x,E(x,r))\
 =\ E(x,F(x,r))\ =\ r\ .
 \]

We now define for every $x\in K$
\begin{equation}\label{rhoH}
\rho_x(y)\ =\ F\left(x,\frac{1}{\Gamma(x,y)}\right)\ .
\end{equation}

We emphasize that in a Carnot group $\bG$ one has, for every $x\in
\bG$, $Q(x) \equiv Q$, the homogeneous dimension of the group, and
therefore the Nagel-Stein-Wainger polynomial is in fact just a
monomial, i.e., $\Lambda(x,r)\equiv C({\bf G}) r^Q$. It follows that
there exists a constant $\omega({\bf G})>0$ such that
\begin{equation}\label{Eg}
 E(x,r)\ \equiv\ \omega({\bf G})\ r^{(Q-p)/(p-1)}\ .
\end{equation}

Using the function $E(x,r)$ in \eqref{E} it should be clear that we
can recast the estimate \eqref{me} in the following more suggestive
form
\begin{equation}\label{me2}
\frac{C}{E(x,d(x,y))}\ \leq\ \gp(x,y)\ \leq \
\frac{C^{-1}}{E(x,d(x,y))}\ .
\end{equation}

As a consequence of \eqref{me2} and of \eqref{rhoH}, we obtain the
following estimate: \emph{there exist positive constants $C, R_0$,
depending on $X_1,...,X_m$ and $K$, such that for every $x\in K$,
and every $0<r\leq R_0$, one has for} $y\in B(x,r)$
\begin{equation}\label{pgH} C\ d(x,y)\ \leq\
\rho_x(y)\ \leq\ C^{-1}\ d(x,y)\ .
\end{equation}

We can thus think of the function $\rho_x$ as a \emph{regularized
pseudo-distance} adapted to the nonlinear operator $\mathcal L_p$.
We denote by \[ B_X(x,r)\ =\ \{y\in \Rn \mid \rho_x(y) < r\}\ ,
\]
the ball centered at $x$ with radius $r$ with respect to the
pseudo-distance $\rho_x$. Because of \eqref{pgH} it is clear that
\[
B(x,C r)\ \subset \ B_X(x,r) \ \subset \ B(x,C^{-1} r)\ .
\]

Our main assumption is that for any $p>1$ the fundamental solution
of the operator $\mathcal L_p$ satisfy the following

\medskip

\noindent \textbf{Hypothesis:}
 \emph{For any compact set $K\subset \Om\subset \Rn$ there exist $C>0$ and $R_0>0$, depending on $K$ and $X_1,...,X_m$,
such that for every $x\in \Om$, $0<R<R_0$ for which
$B_X(x,4R)\subset \Om$ , and a.e. $y \in B(x,R)\setminus\{x\}$ one
has}
\begin{equation}\label{Hyp}
|X \gp(x,y)|\ \leq\ C^{-1}\
\left(\frac{d(x,y)}{\Lambda(x,d(x,y))}\right)^{\frac{1}{p-1}}\ .
\end{equation}

\medskip

We mention explicitly that, as a consequence of the results in
\cite{NSW} and \cite{SC}, the assumption \eqref{Hyp} is fulfilled
when $p=2$. For $p\not=2$ it is also satisfied in any Carnot group
of Heisenberg type $\bf G$. This follows from the results in
\cite{CDG3}, where for every $1<p<\infty$ the following explicit
fundamental solution of $-\mathcal L_p$ was found:
\begin{equation}\label{fsHtype}
 \gp(g)\ =\ \begin{cases} \frac{p-1}{Q-p}
 \sigma_p^{-\frac{1}{p-1}}\ N(g)^{- \frac{Q-p}{p-1}}\
 ,\quad\quad\quad p \not= Q\ ,
 \\
 \\
\sigma_Q^{-\frac{1}{Q-1}}\ \log N(g)\ ,\quad\quad\quad p = Q\ ,
 \end{cases}
 \end{equation}
where we have denoted by $N(g) = (|x(g)|^4 + 16
|y(g)|^2)^{\frac{1}{4}}$ Kaplan's gauge on $\bG$, see \cite{K}, and
we have let $\sigma_p = Q \omega_p$, with
\[
\omega_p\ =\ \int_{\{g\in \bG\mid N(g)<1\}} |X N(g)|^p\ dg\ .
\]

We note that the case $p=2$ of \eqref{fsHtype} was first discovered
by Folland \cite{F1} for the Heisenberg group, and subsequently
generalized by Kaplan \cite{K} to groups of Heisenberg type. The
conformal case $p = Q$ was also found in \cite{HH}.

We stress that the hypothesis \eqref{Hyp} is not the weakest one
that could be made, and that to the expenses of additional
technicalities, we could have chosen substantially weaker
hypothesis.

We now recall the classical one-dimensional Hardy inequality
\cite{H}: let $1<p<\infty$, $u(t)\geq 0$, and $\varphi(t) = \int_0^t
u(s) ds$, then
\[
\int_0^\infty \left(\frac{\varphi(t)}{t}\right)^p dt\ \leq\
\left(\frac{p}{p-1}\right)^p\ \int_0^\infty \varphi'(t)^p dt\ .
\]

Here is our main result.

\begin{theorem}\label{T:HardyH}
Given a compact set $K\subset \Om\subset \Rn$, let $x\in K$, and
suppose that $1<p<Q(x)$. For any $0<R<R_0$ such that
$B_X(x,4R)\subset \Om$ one has for $\varphi\in S^{1,p}_0(B_X(x,R))$
\begin{equation*}\label{hHo}
\int_{B_X(x,R)} |\varphi|^{p}
\left\{\frac{E'(x,\rho_x)}{E(x,\rho_x)}\right\}^p\ |X \rho_x|^p\ dy\
\leq\ \left(\frac{p}{p-1}\right)^p\ \int_{B_X(x,R)} |X \varphi|^p\
dy\ .
\end{equation*}
When $\Lambda(x,r)$ is a monomial (thus, e.g., in the case of a
Carnot group) the constant in the right-hand side of the above
inequality is best possible.
\end{theorem}

We do not present here the proof of Theorem \ref{T:HardyH}, but
refer the reader to the forthcoming article \cite{DGP}. Some
comments are in order. First of all, concerning the factor $|X
\rho_x|^p$ in the left-hand side of the inequality in Theorem
\ref{T:HardyH}, we emphasize that the hypothesis \eqref{Hyp} implies
that $X\rho_x\in L^\infty_{\rm{loc}}$. Secondly, as it is shown in
\cite{DGP} one has
\begin{equation}\label{lambda}
 \left(\frac{Q(x) -
p}{p-1}\right)^p \frac{1}{\rho_x^p}\ \leq\
\left\{\frac{E'(x,\rho_x)}{E(x,\rho_x)}\right\}^p\ \leq\
\left(\frac{Q - p}{p-1}\right)^p \frac{1}{\rho_x^p}\ .
\end{equation}

As a consequence of Theorem \ref{T:HardyH} and \eqref{lambda} we
thus obtain the following.

\begin{corollary}\label{C:HardyH}
In the same hypothesis of Theorem \ref{T:HardyH} one has for
$\varphi\in S^{1,p}_0(B_X(x,R))$
\begin{equation}\label{hardyo}
\int_{B_X(x,R)} \frac{|\varphi|^{p}}{\rho_x^p}  |X \rho_x|^p\ dy\
\leq\ \left(\frac{p}{Q(x)-p}\right)^p\ \int_{B_X(x,R)} |X
\varphi|^p\ dy\ .
\end{equation}
\end{corollary}

Thirdly, it is worth observing that, with the optimal constants,
neither Theorem \ref{T:HardyH}, nor Corollary \ref{C:HardyH} can be
obtained from Corollary \ref{LQ1} above.

We mention in closing that for the Heisenberg group $\Hn$ with $p=2$
Corollary \ref{C:HardyH} was first proved in \cite{GL}. The
inequality \eqref{hardyo} was extended to the nonlinear case
$p\not=2$ in \cite{NZW}. For Carnot groups of Heisenberg type and
also for some operators of Baouendi-Grushin type the inequality
\eqref{hardyo} has been obtained in \cite{Dam}. In the case $p=2$
various weighted Hardy inequalities with optimal constants in groups
of Heisenberg type have also been independently established in
\cite{Ko}. An interesting generalization of the results in
\cite{NZW}, along with an extension to nilpotent Lie groups with
polynomial growth has been recently obtained in \cite{Lo}. In this
latter setting an interesting form of the uncertainty principle
 connected to the case $p=2$ of the Hardy type inequality \eqref{hardyo} has been
established in \cite{CRS}. These latter two references are not
concerned however with the problem of finding the sharp constants.

\end{document}